\documentclass[reqno]{amsart}%11pt
\usepackage[top=2cm, bottom=2cm, outer=3cm, inner=3cm]{geometry}
\usepackage{amssymb,color}
\usepackage{amsmath}
\usepackage{amsfonts}
\usepackage{fouridx}
\usepackage[normalem]{ulem}
\usepackage{soul}
\usepackage[pdftex]{hyperref}
\usepackage{relsize}
\usepackage[utf8]{inputenc}
\usepackage{mathrsfs}
\usepackage[english]{babel}
\usepackage[totoc]{idxlayout}
\usepackage{fancyhdr}
\usepackage{paralist}
\usepackage{xcolor}
\usepackage[mathscr]{euscript}
\usepackage{mathtools}

\usepackage{todonotes}

\allowdisplaybreaks[3]
\newcommand\delc[1]{}

\newcommand\comcd[1]{}

\newcommand\del[1]{}
\newcommand\deln[1]{}
\newcommand\delr[1]{}

\newcommand\comad[1]{}

\newcommand\Greendel[1]{}
\newcommand\old[1]{}

\numberwithin{equation}{section}

\def\R{{\mathbb R}\,}

\newcommand{\sO}{\mathscr{O}}

\def\old#1{}

\def\text#1{{\rm #1}}
\def\newold#1{}

\theoremstyle{plain}
\newtheorem{theorem}{Theorem}[section]
\theoremstyle{remark}
\newtheorem{remark}[theorem]{Remark}

\theoremstyle{plain}

\newtheorem{lemma}[theorem]{Lemma}
\newtheorem{proposition}[theorem]{Proposition}
\newtheorem{definition}[theorem]{Definition}

\newtheorem{assumption}[theorem]{Assumption}

\numberwithin{equation}{section}

\begin{document}
\title[Stochastic Swift-Hohenberg Equation on Hilbert Manifold]{On the Global solution and Invariance of stochastic constrained Modified Swift-Hohenberg Equation on a Hilbert manifold}
\author{Javed  Hussain}
\address{Department of Mathematics\\
Sukkur IBA University\\
Sindh Pakistan}
\email{javed.brohi@iba-suk.edu.pk} 

\author{Saeed Ahmed}
\address{Department of Mathematics\\
Sukkur IBA University\\
Sindh Pakistan}
\email{saeed.msmaths21@iba-suk.edu.pk}

\author{Abdul Fatah}
\address{Department of Mathematics\\
Sukkur IBA University\\
Sindh Pakistan} 

\date{\today}
\begin{abstract}
This paper aims to investigate the stochastic generalization of the projected deterministic constrained modified Swift-Hohenberg equation. In particular, we prove the global well-posedness and its invaraince of Hilbert submanifold i.e. if the initial condition are chosen from submanifold then trajectories of solutions are going to stay on manifold. Swift-Hohenberg equations belong to class of Amplitude equations that usually describe the pattern formation in nature.  
\end{abstract}

\maketitle

\baselineskip 12pt

\section{Introduction}

In this paper, we are  interested in the following stochastic-constrained Modified Swift-Hohenberg evolution equation with  the Stratonovich noise.

 \begin{align}{\label{main_Prb_Intr}}
        du &=\pi _{u}(-\Delta^{2}u+2\Delta u -au - u^{2n-1}) ~dt +  \sum_{k=1}^{N} B_{k}(u) \circ dW_{k}\\ \notag
u(0,x) &= u_{0}.  \notag \\
u(t,x) &= 0, ~~~~~~~~ on ~~~ x \in \partial \mathcal{O}, \notag
    \end{align}

where $\mathcal{O} \subset \mathbb{R}^{2}$ is smooth, continuous and bounded domain, $\left(W_{k}\right)_{k=1}^{N}$  is the  $N$ dimensional, $\mathbb{R}^{N}$-valued {Brownian motion} on the $\left( \Omega, \mathbb{F}, \left(\mathcal{F}_{t}\right)_{t\geq 0}, \mathbb{P}\right) $  filtered probability space. For the fixed elements $f_{1}, f_{2}, f_{3} , ... f_{N}$ in ${H}_{0}^{1}(\mathcal{O}) \cap {H}^{2}(\mathcal{O})$, the map $B_{k}: {H}_{0}^{1}(\mathcal{O}) \cap {H}^{2}(\mathcal{O}) \rightarrow {H}_{0}^{1}(\mathcal{O}) \cap {H}^{2}(\mathcal{O})$ are defined as:
\begin{align}{\label{B_K_st_intr}}
    B_{k}(u) &= \pi_{u}(f_{k})=f_{k}- \langle f_{k},u\rangle u,~~~~~~k= 1,2,3,...N 
\end{align} 
To achieve some geometric properties of solution, such as $\mathcal{L}^2(\mathcal{O})$ invariance, we have taken the noise term to be of the Stratonovich type (see \cite{Hussain_2005}). For initial conditions, we suppose that $u_{0} \in {{H}_{0}^{1}(\mathcal{O}) \cap {H}^{2}(\mathcal{O})} \cap {M}$, where $M$ is following submanifold in $\mathcal{L}^2(\mathcal{O})$,
\begin{align*}
    {M} =  \{u \in \mathcal{L}^2(\mathcal{O}):|u|_{\mathcal{L}^2(\mathcal{O})}^{2}=1\}.     
\end{align*}

In study of pattern formation, modified Swift-Hohenberg equation plays an important role \cite{Maria.B.kania, Qu,Dirk}. Connected with Rayleigh–Bénard convection, it has been employed to address a variety of problems, such as Taylor–Couette flow \cite {PC,Pomeau} and in the study of lasers \cite{Lega}. In addition to other areas of scieite research, it is a valuable tool in material science. This elucidates the surface morphologies during crystal growth \cite{5}, self-assembly processes \cite{6}, and phase transitions \cite{7}. Deterministic forms model regular patterns, whereas stochastic versions incorporate randomness, capturing thermal fluctuations \cite{9} and growth uncertainties \cite{8}. This equation facilitates a comprehensive understanding of intricate material behaviors, enabling advancements in the field of thin-film film deposition \cite{10} and photonic materials   \cite{11}. At that time, the focus was on the global attractor, stability of stationary solutions, and pattern selections of solutions of the deterministic Modified Swift-Hohenberg equation \cite{Peletier1, Peletier2,Peletier3}. However, in recent years, there has been a growing interest in Stochastic Swift-Hohenberg.  Stochastic models are more realistic as noise models the small irregular fluctuations produced by the microscopic effects. The approximation representation of parameterizing manifold and non-Markovian reduced systems for a stochastic Swift–Hohenberg equation with additives was analyzed in \cite{Guo2}. The
results for approximation of manifolds for stochastic Swift–Hohenberg equation with multiplicative noise in Stratonovich sense can been seen in \cite{Dirk, Dirk1, Lin, Swift}. A rigorous error estimation verification of the existence of an amplitude equation for the stochastic Swift–Hohenberg equation was provided by Klepeal et al. \cite{Klepel}. The dynamics and invariant manifolds for a nonlocal stochastic Swift–Hohenberg equation with multiplicative noise were presented in \cite{Guo}. However, to the best of our knowledge, there are no prior works on the stochastic generalization of projected deterministic constrained modified Swift-Hohenberg equation (\ref{main_Prb_Intr}). We aim to fill the gap in this paper.\\

Our work extends the research presented in section 3 of the thesis \cite{Hussain_2005}.  In this paper, we will prove the global wellposednes the  problem (\ref{main_Prb_Intr}) using the Khashminskii test for non-explosions along with the invariance of the solution on the Hilbert manifold.\\

The structure of the paper is as follows: Section 2 is devoted to functional settings, definitions, and the existence and uniqueness of a local mild solution to both the approximated and main stochastic evolution equation. Section 3 presents an Amalgamation Lemma. Section 4 focuses on the no-explosion result and global solution to the proposed problem (\ref{main_Prb_Intr}).

\section{Functional Settings}  

Assume that $\mathcal{O} \in \mathbb{R}^{2}$ is the bounded continuous domain, and for any $p \in [0, \infty), \mathcal{\mathcal{L}}^{P}(\mathcal{O})$ is the Banach space of the Lebesgue measurable function that takes values in $\R$-and  $p$-th power integrable. The norm on that space is given  by

\begin{align*}
    \left|u\right|^{p}_{\mathcal{\mathcal{L}}^{P}(\mathcal{O})}= \int_{\mathcal{O}}{|u(s)|^{p}} ds,~~~~~~~~~~u \in \mathcal{\mathcal{L}}^{P}(\mathcal{O})
\end{align*}
In particular, considering $p=2$, the space $\mathcal{\mathcal{L}}^{2}(\mathcal{O})$ is a Hilbert space with an inner product $\langle ., . \rangle_{\mathcal{\mathcal{L}}^{2}(\mathcal{O})}$. \\
Sobolev spaces are denoted by ${W}^{k,p}(\mathcal{O})$, where $p \in [0, \infty)$ with  $u \in \mathcal{\mathcal{L}}^p(\mathcal{O})$ and its  weak derivative $D^{\alpha}u\in \mathcal{\mathcal{L}}^p(\sO), \text{~such ~that~} \, |\alpha|\leq k$.

In a particular case, if $p=2$, ${W}^{k,2}(\mathcal{O})$ is denoted by ${H}^{k}$. For $k=1$, space ${H}^{1}$ is a Hilbert space with the norm given by

\begin{align*}
    \langle u_{1}, u_{2}\rangle_{{H}^{1}} :=\langle u_{1}, \mathrm{u_{2}}\rangle_{\mathcal{\mathcal{L}}^{2}(\sO)}+\langle\nabla u_{1}, \nabla u_{2}\rangle_{\mathcal{\mathcal{L}}^{2}(\sO)}, \quad u_{1}, u_{2} \in {H}^{1}(\sO ).
\end{align*}

 The operator 
 $\mathcal{A} :  D(\mathcal{A}) \rightarrow {\mathcal{L}}^{2}(\mathcal{O})$ is given as

\begin{align}{\label{operator-A}}
D(\mathcal{A}) =  {H}_{0}^{1}(\mathcal{O}) \cap {H}^{2}(\mathcal{O}) \cap {H}^4(\mathcal{O}),~~~~~~ \mathcal{A}u = \Delta^{2}u-2\Delta u,~ u \in D(\mathcal{A})
\end{align}

The following calculation shows that $\mathcal{A}$ is a self-adjoint operator:

Take the two elements $u_{1}$ and $u_{2}$  in $D(\mathcal{A})$, by using the definition of an inner product and integration by parts \cite{Brezis_2010}; we can deduce that
\begin{align*}
\langle \mathcal{A}u_{1}, u_{2}  \rangle 
        &= \langle \Delta^{2} u_{1}-2\Delta u_{1} , u_{2}  \rangle \\
        &= \langle \Delta^{2}u_{1} , u_{2}  \rangle - 2 \langle \Delta u_{1} , u_{2}  \rangle \\
        &=  \langle \Delta u_{1} , \Delta u_{2}  \rangle + 2 \langle \nabla u_{1} , \nabla u_{2}  \rangle \\
        &=\langle  u_{1} , \Delta ^{2} u_{2}  \rangle - 2 \langle  u_{1} , \Delta u_{2}  \rangle \\ 
       &= \langle  u_{1} , \left(\Delta ^{2} -2 \Delta \right)u_{2} \rangle   \\ 
       &= \langle u_{1}, Au_{2}  \rangle
\end{align*}

\begin{remark}
   The space $(\mathcal{E}, \|.\|) $, $(\mathcal{V},\|.\|_{\mathcal{V}}) $  and $(\mathcal{H}, |.|) $ are denoted as
    \begin{eqnarray*}  
\mathcal{H}:= \mathcal{\mathcal{L}}^2(\mathcal{O}), ~~
\mathcal{V} := {H}_{0}^{1}(\mathcal{O}) \cap {H}^{2}(\mathcal{O}),~~\text{and}~~
\mathcal{E} := \mathcal{D}(\mathcal{A}) = {H}_{0}^{1}(\mathcal{O}) \cap {H}^{2}(\mathcal{O}) \cap {H}^4(\mathcal{O}).
\end{eqnarray*} 
and are dense and continuous, that is, 

\begin{align}{\label{Ass_2.2.2_St}}
    \mathcal{E} \hookrightarrow \mathcal{V} \hookrightarrow  \mathcal{H}
\end{align}
 These notations will be used throughout this paper. 
\end{remark}

Let $\mathcal{L}(\mathcal{X},Y)$  be the space of linear operators that are bounded  from Banach space $\mathcal{X}$ to the Banach space $Y$ ,then we use $ \mathcal{\mathcal{L}}^{2}(a,b;\mathcal{X})$ defined for $[a,b]$, where $0 \leq a < b$, as the space of all equivalence classes of measurable function whose values are in separable Banach space $\mathcal{X}$ with the norm give as: 
\begin{align*}
\left\vert u\right\vert _{\mathcal{\mathcal{L}}^{2}(a,b;\mathcal{X})}=\left(\int_{a}^{b}\left\vert u(p)\right\vert_{\mathcal{X}}^{2}dp\right)^{\frac{1}{2}} < \infty
\end{align*} 
And for $0 \leq a < b$ we introduce,
\begin{align*}
   \mathcal{X}_{a,b}:=\mathcal{\mathcal{L}}^{2}(a,b;\mathcal{E}) \cap C\left( \left[ 0,T\right] ;\mathcal{V}\right) ,
\end{align*} 
then it can be proven that  $ \left(\mathcal{X}_{a,b}, \left|\cdot\right|_{\mathcal{X}_{a,b}} \right) $  is also a Banach space with norm given by:

\begin{align*}
    \left\vert u\right\vert _{\mathcal{X}_{a,b}}^{2}=\underset{t\in \lbrack a,b]}{\sup }\left\Vert u(t)\right\Vert ^{2}
+\int_{a}^{b}\left\vert u(p)\right\vert_{\mathcal{E}}^{2}dp
\end{align*}

For $a=0$ and $b=T$ we introduce $\mathcal{X}_{T}=\mathcal{X}_{0,T}$ and the map $t \rightarrow |u|_{\mathcal{X}_{t}}$ is the increasing function.\\
 
Assume that $\left( \Omega, \mathbb{F}, \left(\mathcal{F}_{t}\right)_{t\geq 0}, \mathbb{P}\right) $ is a filtered probability space with probability space $\left( \Omega, \mathbb{F}, \mathbb{P}\right) $ and $\mathbb{F}$ the filtration $\left(\mathcal{F}_{t}\right)_{t\geq 0}$. 
And it satisfies the usual conditions. For a natural number $N$, $(W_{k}(t))^{N}_ {k=1}$ is the  $\mathbb{F}$-Wiener process, for  ~ $t \geq 0$, which takes values in $\mathbb{R}^{N}$. In this dissertation, and specifically in this section, we use the notation $\mathcal{M}^{2}(\mathcal{X}_{T})$ to represent the set of all $\mathcal{E}$-valued progressively measurable processes $u$, where each trajectory of $u$ is almost surely contained in $\mathcal{X}_{T}$. The norm on $\mathcal{M}^{2}(\mathcal{X}_{T})$ is given as:

\begin{align}
    \left|u \right|^{2}_{\mathcal{M}^{2}(\mathcal{X}_{T})} = \mathbb{E}\left( \left|u \right|^{2}_{\mathcal{X}_{T}}\right)= \mathbb{E}\left( \underset{p\in \lbrack 0,T]}{\sup }\left\Vert u(p)\right\Vert ^{2}
+\int_{0}^{T}\left\vert u(p)\right\vert_{\mathcal{E}}^{2}dp\right)
\end{align}

\subsection{Hilbert manifold, tangent space and orthogonal projection}  

The Hilbert manifold $ {M} =  \{~ u \in \mathcal{H}, |u|_{\mathcal{H}}^{2}=1~\}  $ of the Hilbert space $\mathcal{H}$ is going to be discussed throughout the paper. The tangent space is given as $T_{u}{M}= \{~ h: ~~\langle h, u \rangle = 0~ ,~~\forall ~h \in \mathcal{H}\}$, additionally the map  $ \pi_{u}:\mathcal{H} \longrightarrow T_{u}{M}$ is the orthogonal projection onto $u$  and is given by: 
\begin{equation}{\label{lemma_Tangent}}
    \pi_{u}(h)= h-\langle h, u \rangle ~u, ~~~~~~h \in \mathcal{H}.
\end{equation}  

By considering $ u \in \mathcal{E} \cap {M}$ and applying the definition of orthogonal projection (\ref{lemma_Tangent}), the projection of $-\Delta^{2}u+2\Delta u -au - u^{2n-1} $ under the map $\pi_{u}$ using integration by parts  \cite{Brezis_2010} can be calculated as:\\
$\pi _{u}(-\Delta^{2}u+2\Delta u -au - u^{2n-1})$
 \begin{align}
  &=-\Delta^{2}u+2\Delta u -au - u^{2n-1}
+\langle \Delta^{2}u-2\Delta u +au + u^{2n-1}, u \rangle ~u \notag \\
&=-\Delta^{2}u+2 \Delta u -au - u^{2n-1} + \langle \Delta^{2}u, u \rangle ~u -2\langle \Delta u, u \rangle ~u \notag \\ &  ~+a\langle u, u  
\rangle ~u+\langle u^{2n-1}, u \rangle ~u \notag \\
&=-\Delta^{2}u+2 \Delta u -au - u^{2n-1} + \langle \Delta u,  \Delta u \rangle ~u -2\langle - \nabla u, \nabla u \rangle ~u \notag \\ &  ~+a\langle u, u  
\rangle ~u+\langle u^{2n-1}, u \rangle ~u \notag \\
&=-\Delta^{2}u+2 \Delta u -au - u^{2n-1} + \| \Delta u\|^{2}_{{\mathcal{L}}^{2}(\mathcal{O})} ~u + 2\| \nabla u\|^{2}_{{\mathcal{L}}^{2}(\mathcal{O})} ~u \notag \\ &  ~ +au+\| u\|^{2n}_{{\mathcal{L}}^{2n}(\mathcal{O})} u \notag \\ 
&=-\Delta^{2}u+2 \Delta u  + \|  u\|^{2}_{\mathcal{H}^{2}_{0}} ~u + 2\|  u\|^{2}_{\mathcal{H}^{1}_{0}} ~u  ~ +\| u\|^{2n}_{{\mathcal{L}}^{2n}} u- u^{2n-1}
\end{align}    

 \subsection{Main and approximated Stochastic evolution equation}

 Assume that $\mathcal{E}$, $\mathcal{H}$ and $\mathcal{V}$ are Hilbert spaces and satisfy the assumption (\ref{Ass_2.2.2_St})  then the following stochastic evolution equation will be discussed in this section.
    \begin{align}{\label{main_Prb_St}}
        du &=\pi _{u}(-\Delta^{2}u+2\Delta u -au - u^{2n-1}) ~dt +  \sum_{k=1}^{N} B_{k}(u) \circ dW_{k}\\ \notag
        &= (-\mathcal{A} u+F(u) ) ~dt +  \sum_{k=1}^{N} B_{k}(u) \circ dW_{k}\\ \notag
u(0) &= u_{0},  \notag 
    \end{align}
where  the function  $ F :  \mathcal{V}\longrightarrow {\mathcal{H}}$ is a map defined as $F(u)=\|  u\|^{2}_{\mathcal{H}^{2}_{0}} ~u + 2\|    u\|^{2}_{\mathcal{H}^{1}_{0}} ~u  +\| u\|^{2n}_{{\mathcal{L}}^{2n}} u- u^{2n-1} $ and  $n \in \mathbb{N}$  (or, in a general sense, a real number such that $n>\frac{1}{2}$) and $u_{0} \in \mathcal{V} \cap {M}$. And for the fixed elements $f_{1}, f_{2}, f_{3} , ... f_{N}$ in $\mathcal{V}$, the map $B_{k}: \mathcal{V}\rightarrow \mathcal{V}$ is defined as:
\begin{align}{\label{B_K_st}}
    B_{k}(u) &= \pi_{u}(f_{k})=f_{k}- \langle f_{k},u\rangle u,~~~~~~k= 1,2,3,...N
\end{align}
 Due to the constrained condition given by manifold, the noise term of the above stochastic differential equation (\ref{main_Prb_St}) contains a term of Stratonovich type (see \cite{ZB}).    
 To the equation  (\ref{main_Prb_St}) into Itô's form, the  Stratonovich term can be written as:
 \begin{align*}
     B_{k}(u) \circ dW_{k} &=  B_{k}(u) dW_{k} + \frac{1}{2} d_{u}B_{k}(B_{k}(u)) dt
 \end{align*}
    
 Therefore the equation (\ref{main_Prb_St}) can be written as:

 \begin{align}{\label{main_eq_st_Ito}}
      du &= \left[-\Delta^{2}u+2 \Delta u+F(u) + \frac{1}{2}\sum_{k=1}^{N} m_{k}(u) \right] ~dt +  \sum_{k=1}^{N} B_{k}(u) dW_{k} \\ \notag 
      u& = u_{0} \notag 
 \end{align}
    Where 
    \begin{align}{\label{m_K_st}}
        m_{k}(u)=d_{u}B_{k}(B_{k}(u)), ~~~ \forall u \in \mathcal{H},~~~~ \text{and}~~~~k=1,2,3,...N
    \end{align}

Now we define the auxiliary function.
 Let $\theta : [0,1] \longrightarrow \mathbb{R}^{+}$ be a function with a compact support and is a non-increasing function such that: 
\begin{align}{\label{truncted_st}}
    \begin{cases}   
\theta(x) = 1, ~~ \text{iff} ~~x \in [0,1] \\
\theta(x) = 0, ~~ \text{iff}~~ x \in [2,\infty) \\
\inf_{ x \in \mathbb{R}^{+}} \theta'(x)  \geq -1 
\end{cases} 
\end{align}
 For $m \geq 1$,  ~~~~~~       $ \theta_{m}(\cdot) = \theta (\frac{\cdot}{m})$

\begin{lemma}
    (\cite{BHP13}, page 57) Assume that $g: \mathcal{R}^{+} \longrightarrow \mathcal{R}^{+}$ is a non-decreasing function, for every $x_{1} x_{2}\in R$ 
\begin{eqnarray} {\label{tr_2}}
    \theta_{m}(x_{1}) g(x_{1}) \leq g(2m) ~~~~\text{and}~~~~
    |\theta_{m}(x_{1})-\theta_{m}(x_{2})|\leq \frac{1}{m}|x_{1}-x_{2}|
\end{eqnarray}

\end{lemma}

In order to show the existence and the uniqueness of local mild solution of the equation (\ref{main_Prb_St}), we will first examine the approximated stochastic evolution equation mentioned below:

\begin{align}{\label{Appr_evolution_eq_st}}
    u^{m}(t) &= S(t)u_{0}+ \int_{0}^{t} {S(t-p)~\theta_{m}\left(\left|u^{m} \right|_{\mathcal{X}_{p}}\right)~ F(u^{m}(p))} dp \\ \notag
    &+ \frac{1}{2} \sum_{k=1}^{N} \int_{0}^{t} {S(t-p)~\theta_{m}\left(\left|u^{m} \right|_{\mathcal{X}_{p}}\right)~ m_{k}(u^{m}(p))} dp \\ \notag 
    &+  \sum_{k=1}^{N} \int_{0}^{t} {S(t-p)~\theta_{m}\left(\left|u^{m} \right|_{\mathcal{X}_{p}}\right)~ B_{k}(u^{m}(p))} d W_{k}(p), ~~~~~~ t\in [0,T]
    \end{align}
  Where   $ \{S(t), 0 \leq t < \infty \}$ is a analytic Semi-group of bounded linear operator $A$, defined in (\ref{operator-A}),  on $\mathcal{H}$.

The following assumptions are very useful and will be implemented through this section.

\begin{inparaenum} 
\begin{assumption}

\item[i) ] For each $T>0$ and $f \in \mathcal{\mathcal{L}}^{2}\left( 0,T;{\mathcal{H}}\right)$ a map $u=S*f $ is defined by:
\begin{eqnarray*}
    u(t) = \int^{T}_{0} {S(t-p)f(p)}~dp,~~~~~ t \in [0,T]
\end{eqnarray*} 

And  $u(t) \in \mathcal{X}_{T}$  such that 
\begin{eqnarray*}
  |u|_{\mathcal{X}_{T}} \leq k_{1} |f|_{\mathcal{\mathcal{L}}^{2}\left( 0,T;{H}\right)}
\end{eqnarray*}

Where $S* :  \mathcal{\mathcal{L}}^{2}\left( 0,T;{H}\right)  \rightarrow \mathcal{X}_{T}$. \\
\\

\item[ii) ] For all $T>0$ and  the process $\xi \in\mathcal{M}^{2}(0,T;\mathcal{H})$,  a map $u=S \odot \xi $ is defined by:
\begin{eqnarray*}
    u(t) = \int^{T}_{0} {S(t-p)~\xi(p)}~dW(p),~~~~~ t \in [0,T]
\end{eqnarray*} 

And  $u(t) \in\mathcal{M}^{2}(\mathcal{X}_{T})$  such that 
\begin{eqnarray*}
  |u|_{\mathcal{M}^{2}(\mathcal{X}_{T})} \leq C_{0} |\xi|_{\mathcal{M}^{2}\left( 0,T;\mathcal{V}\right)}
\end{eqnarray*}
where $ S\odot :   \mathcal{M}^{2}\left( 0,T;\mathcal{V} \right) \rightarrow  \mathcal{M}^{2}(\mathcal{X}_{T}) $ a linear and bounded map.
\item[iii) ] For every $T>0$ and $u_{0} \in \mathcal{V}$ the function $u=S(\cdot)u_{0} $ is given by:
\begin{eqnarray*}
    u(t) = S(t) u_{0},~~~~~ t \in [0,T] 
\end{eqnarray*} 
is in $\mathcal{X}_{T}$ and 
\begin{eqnarray*}
  |u|_{\mathcal{X}_{T}} \leq k_{2} \|u_{0}\|_{\mathcal{V}}
\end{eqnarray*}

\end{assumption}
\end{inparaenum}

Before introducing a definition of the solution, let us define the stopping times and admissible process $\Omega_t(\tau)= \{\omega \in \Omega: t < \tau (\omega)\}$.

\begin{definition}\label{def-adapted process}
 Suppose that  $\mathcal{X}$ is a {topological space}. The process $u : [0,\tau) \times \Omega \to \mathcal{X}$ (called local process), taking values in $\mathcal{X}$, is admissible if and only if:

		\begin{inparaenum}
			\item[i) ] it is $\mathbb{F}$-adapted, that is,  $u\big\vert_{\Omega_t(\tau)}: \Omega_t(\tau) \to
			\mathcal{X}$ is $\mathcal{F}_t$-measurable, $\forall ~t\ge 0$;\\ \item[ii)]
			 The map $[0,
			\tau(\omega))\ni t \mapsto \eta(t, \omega) \in \mathcal{X}$ is continuous for almost all $\omega \in \Omega$.
			\end{inparaenum} \\

For two local processes  $u_{1}: [0,\tau_{1})\times \Omega  \to \mathcal{X}$, and $u_{2}: [0,\tau_{2})\times \Omega  \to \mathcal{X}$, if 
\begin{align*}
  u_{1}(\cdot,\omega)= u_{2}(\cdot,\omega) \mbox{ on } [0,t], ; \mbox{ 		for a.e. $\omega \in \Omega_t(\tau_{1})\cap \Omega_t(\tau_{2})$}.
\end{align*}
then $u_{1}$ and $u_{2}$ are called equivalent processes. And can be expressed as $(u_{1},\tau_{1}) \sim (u_{2},\tau_{2})$   iff $\tau_{1}=\tau_{2}$  $\mathbb{P}$-a.s. and,   for all  $t>0$.\\
Notice: { Two local} admissible processes  $u_{1}: [0,\tau_{1})\times \Omega  \to \mathcal{X}$, and $u_{2}: [0,\tau_{2})\times \Omega  \to \mathcal{X}$ are equivalent if $u_{1}(t)\big\vert_{\Omega_t(\tau_{1})}= u_{2}(t)\big\vert_{\Omega_t(\tau_{2})},~~t>0$  $\mathbb{P}$-a.s.
	\end{definition}

\begin{definition}
    
A stopping time $\tau$ is known as accessible if there is a sequence $\left(\tau_{m}\right)_{m \in \mathbb{N}}$  of increasing stopping times. This sequence acts as an approximation for $\tau$, satisfying the conditions $\tau_{m} < \tau$ on the set $ \{\tau > 0\}$ and $ \lim_{m\to \infty} \tau_{m}= \tau$. \\
We will now define the concepts of a local mild solution, local maximal solution, and global solution for the main problem (\ref{main_Prb_St}).
 \end{definition}
\begin{definition} Let $u_{0}$ is the $\mathcal{V}$ valued and $F_{0}$ measurable random process with $\mathbb{E}\left(\|u_{0}\|^{2}\right) < \infty$, then a pair $ \left(u, \tau\right)$ is a \textbf{local mild solution} to the main problem (\ref{main_Prb_St}) if the following conditions are meet:\\
 \begin{inparaenum}
     \item[i)] $\tau$ is an accessible stopping time.\\
\item[ii)] $u: [0,\tau) \times \omega \rightarrow \mathcal{V}$ is admissible process.\\
\item[iii)] There is a sequence $\left(\tau_{m}\right)_{m \in \mathbb{N}}$  of increasing stopping times, satisfying the conditions $\tau_{m} < \tau$ and $ \lim_{m\to \infty} \tau_{m}= \tau$. And for $m \in \mathbb{N}$ and $t\geq 0$ we have:
 \end{inparaenum}
\begin{align}
    \left|u \right|^{2}_{\mathcal{X}_{{t} \wedge  \tau_{m}}} = \mathbb{E}\left( \underset{q\in \lbrack 0,t \wedge \tau_{m} ]}{\sup }\left\Vert u(q)\right\Vert ^{2}
+\int_{0}^{t \wedge \tau_{m}}\left\vert u(p)\right\vert_{\mathcal{E}}^{2}dp\right) < \infty
\end{align}

and 
\begin{align}
    u(t \wedge \tau_{m}) &= S(t\wedge \tau_{m})u_{0}+ \int_{0}^{t\wedge \tau_{n}} {S(t\wedge \tau_{m}-p)~ F(u(p))} dp \\ \notag
    &+ \frac{1}{2} \sum_{k=1}^{N} \int_{0}^{t} {S(t \wedge \tau_{m}-p)~ m_{k}(u(p))} dp \\ \notag 
    &+  \sum_{k=1}^{N} \int_{0}^{t} {S(\wedge \tau_{m}-p)~ B_{k}(u(p))} d W_{k}(p), ~~~~~~ \mathbb{P}- a.s 
    \end{align}

Where $B_{k}$ and $m_{k}$ are defined in (\ref{B_K_st}) and (\ref{m_K_st}).

\end{definition}

\begin{definition}Assume that $(u, \tau_{\infty})$ represents the local solution to the problem (\ref{main_Prb_St}), and if:

\begin{align*}
    \lim_{t \to \tau_{\infty}} \left|u\right|_{\mathcal{X}_{t}} = \infty ~~~\mathbb{P}- a.s ~~~~ \text{on~the ~set}~~~~\{\omega \in \Omega : \tau_{\infty}(\omega) < \infty \},~~a.s
\end{align*}
In such a case, $(u_{1}, \tau_{1,\infty})$ is referred to as a \textbf{local maximal solution}. If $\tau_{1,\infty} < \infty $ is finite with a positive probability, it is known as explosion time. Additionally, if there exists another local maximal solution $(u_{2}, \tau_{2,\infty})$, the uniqueness of $(u_{1}, \tau_{1,\infty})$ is established when $\tau_{2,\infty}=\tau_{\infty} $ and $u_{1}=u_{2}$  on the interval  $[0, \tau_{\infty})$  $\mathbb{P}-a.s.$.

Thus, a local solution $(u, \tau_{\infty})$ is considered a \textbf{global solution} if the explosion time $ \tau_{\infty}= \infty$.

\end{definition}

\subsection{The Existence and Uniqueness of a Local Mild Solution.}

To establish the existence and uniqueness of a local mild solution for our main problem (\ref{main_Prb_St}), firstly, we examine the existence and uniqueness of the solution to the approximate evolution equation (\ref{Appr_evolution_eq_st}). We will consider a fixed positive actual number $T$ for our analysis. Throughout this section, we will work within the abstract $\mathcal{E}, \mathcal{V}$, and $\mathcal{H}$ spaces that satisfy the assumption (\ref{Ass_2.2.2_St}).
In this subsection, our objective is to demonstrate the existence and uniqueness of local mild solutions for both the truncated version (\ref{Appr_evolution_eq_st}) and the original evolution equation (\ref{main_Prb_St}). \\

\subsection{Important Estimates}

The objective of this subsection is to demonstrate that the non-linear functions $F$, $m_{k}$, and $B_{k}$, which are part of the drift and diffusion terms in the primary stochastic evolution equation (\ref{Appr_evolution_eq_st}), possess local Lipschitz properties and fulfil symmetric estimates. This subsection will focus on specific Hilbert spaces $\mathcal{H}$, $\mathcal{V}$, and $\mathcal{E}$, as outlined in (\ref{Ass_2.2.2_St}), for analysis.

%The well known  Gagliardo-Nirenberg-Sobolev inequality is given as:
%\addtocontents{toc}{\protect\setcounter{tocdepth}{1}}
%\subsection{Lemma} 

%about   Gagliardo-Nirenberg-Sobolev inequality

The following lemma offers an estimation for the non-linear term $F$ found in the drift term of the stochastic evolution equation (\ref{main_eq_st_Ito}).

\begin{lemma}
    
Assume that  $\mathcal{E}$, $\mathcal{V}$ and $\mathcal{H}$ satisfy the assumption (\ref{Ass_2.2.2_St}), and $ F :  \mathcal{V}\longrightarrow \mathcal{H}$ be a map defined as 
\begin{eqnarray}{\label{lipschitz_st}}
    F(u)=\|  u\|^{2}_{\mathcal{H}^{2}_{0}} ~u + 2\|    u\|^{2}_{\mathcal{H}^{1}_{0}} ~u  +\| u\|^{2n}_{{\mathcal{L}}^{2n}} u- u^{2n-1} 
\end{eqnarray}
Then $F$ is locally Lipschitz that is 
\begin{eqnarray*}
    |F(u_{1})-F(u_{2})|_{\mathcal{H}} &\leq& {G} ( \|u_{1}\|_{\mathcal{V}},\|u_{2}\|_{\mathcal{V}}) \|u_{1}-u_{2}\|_{\mathcal{V}}, ~~~~ u_{1}, ~u_{2}~\in \mathcal{V},
\end{eqnarray*} 
 
where $ {G}:   \mathbb{R}^{+}\cup \{0\} \times \mathbb{R}^{+}\cup \{0\}\longrightarrow \mathbb{R}^{+}\cup \{0\}$ is bounded linear polynomial map,  

\begin{eqnarray*}
    &&{G} \left(m,n\right) =2C \left(m^{2}+n^{2}+mn\right) \notag \\ &&+C_{n} \left[  \left( \frac{2n-1}{2}\right) \left( m^{2n-1}+n^{2n-1}\right) (m+n)  + \left( m^{2n}+n^{2n}\right) + \left( 1+ m^{2}+n^{2}\right)^{\frac{1}{3}} \right]
\end{eqnarray*}

\end{lemma}

\begin{lemma}

\cite{Hussain_2005} For any $f\in \mathcal{H}$, the map $B: \mathcal{H} \rightarrow \mathcal{H}$ is given as:

\begin{align}{\label{B_{u}_map}}
    B(u)= f- \langle f, u \rangle u~, ~~~~ u \in \mathcal{H} 
\end{align}

And if $u, v \in \mathcal{H} $ then 
\begin{align}
    \left|B(u)-B(v) \right|_{\mathcal{H}} \leq \left|f\right|\left( \left|u\right|+ \left|v\right|\right)\left|u-v\right|
\end{align}

Additionally, if  $f \in \mathcal{V}$ then:
\begin{align}
      \left\|B(u)-B(v) \right\|_{\mathcal{V}} \leq \left\|f\right\|_{\mathcal{V}}\left( \left\|u\right\|_{\mathcal{V}}+ \left\|v\right\|_{\mathcal{V}}\right)\left\|u-v\right\|_{\mathcal{V}}
\end{align}

\end{lemma}
In the next lemma we are introducing  Fr'echet derivative.

 \begin{lemma}
     
\cite{Hussain_2005} Suppose $f\in \mathcal{H}$ and the map $B: \mathcal{H} \rightarrow \mathcal{H}$  given in (\ref{B_{u}_map}), then for each $u\in \mathcal{H}$ the Frechet derivative exists and defined as:

\begin{align}
    d_{u}(B(s))= -\langle f, u \rangle s-\langle f, s \rangle u~, ~~~~ \forall u, s \in \mathcal{H}
\end{align}
 \end{lemma}

\begin{proposition}{\label{K_{j}}}
    
\cite{Hussain_2005} Let $f \in \mathcal{V}$, then for any $ u, v \in \mathcal{H} $  and $ u, v \in \mathcal{V} $ respectively,  \\
the map $\kappa : \mathcal{H}\ni u \rightarrow  d_{u}(B(s)) \in \mathcal{H}$ satisfied the following Lipschitz estimates:

\begin{align}
    \left| \kappa(u)-\kappa(v) \right| \leq 2|f|^{2} \left(|u|^{2}+|v|^{2}+|u||v|\right) \left|u-v\right|
\end{align}

\begin{align}
    \left\| \kappa(u)-\kappa(v) \right\|_{\mathcal{V}} \leq 2\|f\|_{\mathcal{V}}^{2} \left(\|u\|_{\mathcal{V}}^{2}+\|v\|_{\mathcal{V}}^{2}+\|u\|_{\mathcal{V}}\|v\|_{\mathcal{V}}\right) \left\|u-v\right\|_{\mathcal{V}}
\end{align}
\end{proposition}

\subsection{Existence and uniqueness of the local mild solution to the main stochastic evolution equation}

In the previous section of estimates, we observed that the map $F$, $B$ and $\kappa$ were locally Lipschitz satisfying  
some symmetric estimates. Now, this subsection aims to show the existence and uniqueness of the local mild solution to both the truncated (\ref{Appr_evolution_eq_st}) and main stochastic evolution equation(\ref{main_eq_st_Ito}).

\begin{proposition}{\label{pro_approX_evol_eq}}
 ( \cite{Hussain_2005}, page 133) For any given $f_{1},f_{2}, f_{3},...f_{N} $ and $u_{0}\in \mathcal{V}$, we define a map $ \Upsilon_{T,~ u_{0}}^{m} : \mathcal{M}^{2}(\mathcal{X}_{T}) \rightarrow \mathcal{M}^{2}(\mathcal{X}_{T}) $ defined as:

\begin{align}
\Upsilon_{T,~ u_{0}}^{m} (u) = Su_{0} + S \ast \Gamma_{T, F}^{m}(u) + \frac{1}{2} \sum_{k=1}^{N} S \ast \Pi_{\kappa_{k}, T}(u) + \sum_{k=1}^{N} S \odot \Gamma_{B_{k}, T}(u)
\end{align} 
Where  $\kappa_{k}$ and $B_{k}$ are defined in ({\ref{K_{j}}} ) and (\ref{B_{u}_map}) respectively for $k=1,2,3,...,N$, the map $\Gamma_{T, F}^{m} : \mathcal{X}_{T} \rightarrow  \mathcal{\mathcal{L}}^{2}( 0, T;H) $ is defined as:
\begin{align}
    \Gamma_{T, F}^{m}(u) = \theta_{m}\left( |u|_{\mathcal{X}_{t}}\right)F(u(t))
\end{align}
, another map $\Pi_{T, \kappa}^{m} : \mathcal{X}_{T} \rightarrow  \mathcal{\mathcal{L}}^{2}( 0, T;\mathcal{H}) $ is defined as:
\begin{align}
    \Gamma_{T, \kappa}^{m}(u) = \theta_{m}\left( |u|_{\mathcal{X}_{t}}\right)\kappa(u(t))
\end{align}
 the map $\Gamma_{B, T}^{m} : \mathcal{X}_{T} \rightarrow  \mathcal{\mathcal{L}}^{2}( 0, T;\mathcal{H}) $ is given as:
\begin{align}
    \Gamma_{B, T}^{m}(u) = \theta_{m}\left( |u|_{\mathcal{X}_{t}}\right)B(u(t))
\end{align}
  Then there is a constant $C(m)>0$ such that $\forall u , v \in\mathcal{M}^{2}(\mathcal{X}_{T})$ we have:

  \begin{align}
      \left|\Upsilon_{T,~ u_{0}}^{m} (u)-\Upsilon_{T,~ u_{0}}^{m} (v)\right|_{\mathcal{M}^{2}(\mathcal{X}_{T})} \leq C(m)\left| u-v\right|_{\mathcal{M}^{2}(\mathcal{X}_{T})} \sqrt{T}
  \end{align}
   
   Additionally, $\exists$ $T_{0} > 0 $ in such a way that $\forall T \in (0,T_{0}) $,  $ \Upsilon_{T,~ u_{0}}^{m} $ is strict contraction, More precisely,  $\forall T \in (0,T_{0}) $ we have $u \in \mathcal{X}_{T}$ such that:
   \begin{align*}
      \Upsilon_{T,~ u_{0}}^{m} (u) =u
   \end{align*}
   \end{proposition}

\begin{remark}

  By proposition  \ref{pro_approX_evol_eq}, we can conclude that the following approximated stochastic evolution equation. 
   
    \begin{align*}
    u^{m}(t) &= S(t)u_{0}+ \int_{0}^{t} {S(t-p)~\theta_{m}\left(\left|u^{m} \right|_{\mathcal{X}_{p}}\right)~ F(u^{m}(p))} dp \\ \notag
    &+ \frac{1}{2} \sum_{k=1}^{N} \int_{0}^{t} {S(t-p)~\theta_{m}\left(\left|u^{m} \right|_{\mathcal{X}_{p}}\right)~ m_{k}(u^{m}(p))} dp \\ \notag 
    &+  \sum_{k=1}^{N} \int_{0}^{t} {S(t-p)~\theta_{m}\left(\left|u^{m} \right|_{\mathcal{X}_{p}}\right)~ B_{k}(u^{m}(p))} d W_{k}(p), ~~~~~~ t\in [0,T]
    \end{align*}
    
      has a unique solution in $\mathcal{X}_{T}$.\\
    \end{remark}

\begin{proposition}
  (\cite{Hussain_2005}, page 137)   For any $K>0$ and $\epsilon >0$ there is a constant $K^{*}(\epsilon, R)>0$ such that for each $u_{0}$, $\mathcal{F}_{0}$- measurable and $\mathcal{V}$ valued random variable, that satisfies $\mathbb{E} \|u_{0}\|^{2}_{\mathcal{V}} < \infty$. Then there is a unique local solution $\left( u(t), t < \tau \right)$ to the main problem (\ref{main_Prb_St}) such that $ \mathbb{P}(\tau \geq K^{*})\geq 1- \epsilon$.\\
\end{proposition}

\subsection{Global solution to the approximated stochastic evolution equation}

Assume that the stopping time sequence is given as:
\begin{align*}
    \tau_{m}:= \inf \left \{  t \in [0,T]: |u^{m}|_{\mathcal{X}_{t}} \geq m \right \} \wedge T
\end{align*}

\begin{theorem}
    
Assuming that Assumptions (\ref{Ass_2.2.2_St}) and the assumptions stated in Proposition (\ref{pro_approX_evol_eq}) are satisfied, and considering the stopping times sequence $\left( \tau_{m} \right)_{m \in \mathbb{N}}$   as mentioned earlier, we can assert that for each $m \in \mathbb{N}$, the truncated evolution equation (\ref{Appr_evolution_eq_st}) has a unique global solution $u^{m} \in  \mathcal{M}^{2}(\mathcal{X}_{T})$ . Additionally, it should be noted that $(u^{n}, \tau_{n})$ is a local, mild solution to our main problem (\ref{main_eq_st_Ito}).
\end{theorem}

\section{Local maximal mild solution} 

This section aims to construct the local maximal solution to our central stochastic equation (\ref{main_eq_st_Ito}). This can be done with the help of lemma \cite{ZB-beam}.

\begin{lemma}{\textbf{(Amalgamation Lemma)}}  \\
i) Suppose $\Delta$ is a collection of stopping times that take values in $[0,\infty]$. Then, the supremum of $\Delta$, denoted as $\tau = \sup{\Delta}$, is an accessible stopping time that also takes values in $[0,\infty]$ . Furthermore, there exists a sequence $\{\beta_{m}\}_{m=1}^{\infty}$ such that $\tau(\omega)= \lim_{m \to \infty} \beta_{m}(\omega), ~~~~ \forall \omega \in \Omega$. \\

ii) Suppose that for every $ \beta \in \Delta $ the map $ I_{\beta} : [0, \beta) \times \Omega \rightarrow \mathcal{V}$ is an admissible process such that for every $\beta, \beta' \in \Delta $ and $t >0$ we have:
\begin{align}{\label{parti}}
    I_{\beta}(t) = I_{\beta'}(t) ~~~~ on ~~ \Omega_{t}(\beta\wedge \beta') 
\end{align}
then admissible  process $\textbf{I}: [0, \tau) \times \Omega \rightarrow \mathcal{V}$ exists and  for any $t >0 $ we have 
\begin{align}{\label{partii}}
     \textbf{I}(t) = I_{\beta}(t) ~~~~\mathbb{P}- a.s.~~~~ on ~~ \Omega_{t}(\beta) 
\end{align} \\
iii) Furthermore, any process $ {I}_{\tau} :[0, \tau) \times \Omega \rightarrow \mathcal{X}$  that satisfies (\ref{partii}) then such process ${I}_{\tau}$ is the version of $\textbf{I}$ that is; for each $t \in [0, \infty)$ we have 
\begin{align}
    \mathbb{P}\left(\left\{ \omega \in \Omega : t < \tau (\omega), {I}_{\tau}( t, \omega) \neq \textbf{I}(t, \omega)\right\}\right) =0 
\end{align}
\end{lemma} 

More precisely, for any admissible process  ${I}_{\tau}$, we have:  
\begin{align}
    \textbf{I}={I}_{\tau}
\end{align}

\begin{remark}
    We should acknowledge that due to the admissibility of both processes, $ I_{\beta} : [0, \beta) \times \Omega \rightarrow V$ and $\textbf{I}: [0, \tau) \times \Omega \rightarrow V$ (which ensures their trajectories are almost surely continuous), and considering that $\beta$ is less than or equal to $\tau$, condition (\ref{partii}) can be expressed in an equivalent form:
\begin{align}
    \textbf{I}|_{[0, \beta) \times \Omega} = I_{\beta}
\end{align}

In the same way, the condition (\ref{partii}) can be reformulated as follows:

\begin{align}
    I_{\beta'}|_{[0, \beta\wedge \beta') \times \Omega} = I_{\beta}|_{[0, \beta \wedge \beta') \times \Omega}
\end{align} 
 
\end{remark}

\begin{theorem}
  (\cite{Hussain_2005}, page 147)   Suppose that $u_{0}$ be a random variable and taking values in $\mathcal{V}$ and measurable concerning $\mathcal{F}_{0}$. We assume that the following two criteria are met:\\
i) There is  at least one local solution $(u, \tau)$ to the problem (\ref{main_eq_st_Ito}), and \\
ii) if $(u_{1}, \tau_{1})$  and $(u_{2}, \tau_{2})$  are any two local solutions, then for any $t > 0$, $u_{1}(t)$ is almost surely equal to $u_{2}(t)$ on the subset $\Omega_{t}(\tau_{1} \wedge \tau_{2})$ of $~\Omega$. This condition can be expressed as: 
\begin{align}
    u_{1}(t) = u_{2}(t), ~ \mathbb{P}- a.s~~~ on~~ \Omega_{t}(\tau_{1} \wedge \tau_{2})
\end{align}

Under the above assumptions, problem (\ref{main_eq_st_Ito}) possesses a unique maximal local solution $(\hat{u}, \hat{\tau})$ that satisfies the inclusion:
$(u, \tau) \leq  (\hat{u}, \hat{\tau})$.\\
\end{theorem}

\section{No Explosion Result and Global solution } 
In this section, we are proving no explosion results, and then finally, we will show the global solution for our central stochastic evolution equation (\ref{main_eq_st_Ito}).

\begin{theorem}
    (\cite{Hussain_2005}, page 153) For each $\mathcal{F}_{0}$ measurable initial data $u_{0}$  taking values in $\mathcal{V}$- with that satisfies $\mathbb{E}\|u_{0}\|^{2}_{\mathcal{V}}< \infty$. There is a unique maximal solution $(u, \tau_{\infty})$ to a problem  (\ref{main_eq_st_Ito}). And 
\begin{align*}
    \mathbb{P}\left(\left\{ \tau_{\infty} < \infty \right\} \cap \left\{ \sup_{t \in [0,\tau_{\infty})} \|u(t)\|_{\mathcal{V}} < \infty \right\} \right) =0
\end{align*} 
and 
\begin{align*}
    \limsup_{{t \to \tau_{\infty}}} \|u(t)\|_{\mathcal{V}} = \infty , ~~~~a.s.~~~ on ~~~\left\{ \tau_{\infty} < \infty \right\}
\end{align*} 
\end{theorem}

\subsection{Invariance of Manifold}
    
Now, this subsection will demonstrate the manifold's invariance. For any given initial data in submanifold $M$, all trajectories of the solution to the problem (\ref{main_eq_st_Ito})  are also in $M$. This invariance will play a pivotal role in proving the wellposedness of global solutions. Additionally, we will introduce the energy function and related lemma to prove the invariance above.

\begin{lemma}{\label{Lemma_Inva}}

Consider the map $\gamma : \mathcal{H} \rightarrow \mathbb{R}$ defined by:
\begin{align}{\label{l-0}}
    \gamma(u) = \frac{1}{2}|u|_{\mathcal{H}}^{2}, ~~~~ \forall u \in \mathcal{H}
\end{align}

is $C^{2}$- class and for any $u, p, p_{1}, p_{2}  \in \mathcal{H}$ we have:
\begin{align}
d_{u} \gamma(p) &:= \left\langle u, p\right\rangle {\label{D_u(p)}} \\
 d^{2}_{u} \gamma(u) (p_{1}, p_{2})&:= \left\langle p_{1}, p_{2}\right\rangle {\label{l-123}}
\end{align} 
Furthermore, for any $f \in \mathcal{V}$ if 
\begin{align*}
    B(u) = f- \left\langle f, u \right\rangle u, ~~~ \text{and}~~~m(u) = - \left\langle f, B (u) \right\rangle u- \left\langle f, u \right\rangle B(u)
\end{align*}
Then the following equations hold:
\begin{align}
    \left\langle \gamma'(u), B(u)\right\rangle &= \left\langle u, f\right\rangle(|u|_{\mathcal{H}}^{2}-1){\label{lemm_1}} \\ 
    \left \langle \gamma'(u), -\Delta^{2}u+2 \Delta u +F(u)\right\rangle &= \left( \|  u\|^{2}_{\mathcal{H}^{2}_{0}}  + 2\|    u\|^{2}_{\mathcal{H}^{1}_{0}}   +\| u\|^{2n}_{{\mathcal{L}}^{2n}}  \right)(|u|_{\mathcal{H}}^{2}-1) {\label{lemm_prf}}\\
     \gamma''(u) (B(u), B(u)) &= |f|^{2}+ \left\langle f, u\right\rangle^{2} (|u|_{\mathcal{H}}^{2}-2) {\label{lemm_2}}\\
      \langle \gamma'(u), \kappa(u)\rangle &= - |u|_{\mathcal{H}}^{2}|f|^{2}+\left \langle f, u\right\rangle^{2}(2|u|_{\mathcal{H}}^{2}-1){\label{lemm_3}}
\end{align} 
\end{lemma}

\begin{proof}
    The proofs of  (\ref{l-0}), (\ref{D_u(p)}),  (\ref{l-123}), (\ref{lemm_1}),(\ref{lemm_2}) and (\ref{lemm_3}) can be seen in (\cite{Hussain_2005}, page 159). We only need to show that (\ref{lemm_prf}): \\
    
For any $u \in \mathcal{H}$, using  (\ref{D_u(p)})\\
$\left \langle \gamma'(u), -\Delta^{2}u+2 \Delta u +F(u) \right \rangle$
    \begin{align*}
         &= \left\langle u, -\Delta^{2}u+2 \Delta u +\|  u\|^{2}_{\mathcal{H}^{2}_{0}} ~u + 2\|    u\|^{2}_{\mathcal{H}^{1}_{0}} ~u  +\| u\|^{2n}_{{\mathcal{L}}^{2n}} u- u^{2n-1} \right \rangle \\
         & = -\left\langle u, \Delta^{2}u \right\rangle +2 \left \langle u,\Delta u \right \rangle+ \left\langle u, \|  u\|^{2}_{\mathcal{H}^{2}_{0}} ~u \right\rangle+ 2 \left\langle u, \|    u\|^{2}_{\mathcal{H}^{1}_{0}} ~u \right\rangle+ \left\langle u,\| u\|^{2n}_{{\mathcal{L}}^{2n}} u \right\rangle- \left\langle u, u^{2n-1} \right\rangle \\
        &= - \|  u\|^{2}_{\mathcal{H}^{2}_{0}} -2 \|  u\|^{2}_{\mathcal{H}^{1}_{0}}+ \|  u\|^{2}_{\mathcal{H}^{2}_{0}} |u|^{2}_{\mathcal{H}} + \|  u\|^{2}_{\mathcal{H}^{1}_{0}} |u|^{2}_{\mathcal{H}}+ \| u\|^{2n}_{{\mathcal{L}}^{2n}}|u|^{2}_{\mathcal{H}} - \| u\|^{2n}_{{\mathcal{L}}^{2n}} \\
        & = \left(\|  u\|^{2}_{\mathcal{H}^{2}_{0}} + 2 \|  u\|^{2}_{\mathcal{H}^{1}_{0}}+ \| u\|^{2n}_{{\mathcal{L}}^{2n}} \right) \left( |u|^{2}_{\mathcal{H}}-1 \right) 
    \end{align*} 
\end{proof}

In the next subsection, the invariance of a manifold will be proven, which plays a pivotal role in proving the global solution of our main stochastic evolution equation.
In this sense, the stopping time is given as:

\begin{align}
    \tau_{\ell} = \inf \left\{ t \in [0,T] : \|u(t)\|_{\mathcal{V}} \geq \ell\right\}, ~~~ \forall~\ell \in \mathbb{N}
\end{align}

\begin{proposition}
    
Assume that the conditions in the lemma (\ref{Lemma_Inva}) are satisfied. Now if $u_{0} \in{M} $ then $u\left( t \wedge \tau_{\ell} \right) \in{M} $~~~ $\forall t \in [0, T]$
\end{proposition}
\begin{proof}
    assume that $u_{0} \in \mathcal{V} \cap{M}$ and $t \in[0,T] $. \\
    Recall that the main stochastic evolution equation in Itô form is:
    \begin{align*}
      du &= \left[-\Delta^{2}u+2 \Delta u+F(u) + \frac{1}{2}\sum_{k=1}^{N} m_{k}(u) \right] ~dt +  \sum_{k=1}^{N} B_{k}(u) dW_{k}, ~~~~~~~~\mathbb{P}-a.s.
   \end{align*}
    
 By applying the Itô lemma (see \cite{32}) to the map $ \gamma: \mathcal{H}\ni u \rightarrow \frac{1}{2}|u|_{\mathcal{H}}^{2}$, we have:

\begin{align*}
    \gamma \left(u(t \wedge \tau_{\ell})\right) -\gamma(u_{0})&= \sum_{k=1}^{N} \int_{0}^{t \wedge \tau_{\ell}} \langle \gamma'(u(p)) , B_{k}(u(p)) \rangle ~dW_{k}(p) \\
    &+\frac{1}{2} \sum_{k=1}^{N} \int_{0}^{t \wedge \tau_{\ell}} \langle \gamma'(u(p)) , m_{k}(u(p)) \rangle ~dp \\
    &+\frac{1}{2} \sum_{k=1}^{N} \int_{0}^{t \wedge \tau_{\ell}} \gamma''(u(p))\left( B_{k}(u(p)) , B_{k}(u(p)) \right) ~dp\\
    &+ \int_{0}^{t \wedge \tau_{\ell}} \langle \gamma'(u(p)) , -\Delta^{2}u(p)+2 \Delta u(p)+F(u(p)) \rangle ~dp, ~~~~~~~\mathbb{P}-a.s.
\end{align*}

By substituting (\ref{lemm_1}), (\ref{lemm_prf}), (\ref{lemm_2}), and (\ref{lemm_3}), with  $u(p) =u$, it follows that:

\begin{align*}
     \gamma \left(u(t \wedge \tau_{\ell})\right) -\gamma(u_{0})&= \sum_{k=1}^{N} \int_{0}^{t \wedge \tau_{\ell}} \langle u(p), f_{k}\rangle(|u(p)|_{\mathcal{H}}^{2}-1) ~dW_{k}(p) \\
    &+ \sum_{k=1}^{N} \int_{0}^{t \wedge \tau_{\ell}}\left( - |u(p)|_{\mathcal{H}}^{2}|f_{k}|^{2}+ \langle f_{k}, u(p)\rangle^{2}(2|u(p)|_{\mathcal{H}}^{2}-1)\right) ~dp \\
    &+ \sum_{k=1}^{N} \int_{0}^{t \wedge \tau_{\ell}} \left(|f_{k}|^{2}+ \langle f_{k}, u(p)\rangle^{2} (|u(p)|_{\mathcal{H}}^{2}-2) \right)~dp\\
    &+ \int_{0}^{t \wedge \tau_{\ell}} \left( \|  u(p)\|^{2}_{\mathcal{H}^{2}_{0}}  + 2\|    u(p)\|^{2}_{\mathcal{H}^{1}_{0}}   +\| u(p)\|^{2n}_{{\mathcal{L}}^{2n}}  \right)(|u(p)|_{\mathcal{H}}^{2}-1) ~dp, ~~~~~~~\mathbb{P}-a.s
\end{align*}

Using $|u_{0}|^{2}_{\mathcal{H}}=1$, we have

\begin{align*}
   \frac{1}{2}\left(|u(t \wedge \tau_{\ell})|^{2}_{\mathcal{H}}-1\right) &= \sum_{k=1}^{N} \int_{0}^{t \wedge \tau_{\ell}} \langle u(p), f_{k}\rangle(|u(p)|_{\mathcal{H}}^{2}-1) ~dW_{k}(p) \\
    &+\frac{1}{2} \sum_{k=1}^{N} \int_{0}^{t \wedge \tau_{\ell}}\left( - |u(p)|_{\mathcal{H}}^{2}|f_{k}|^{2}+ \langle f_{k}, u(p)\rangle^{2}(2|u(p)|_{\mathcal{H}}^{2}-1)\right) ~dp \\
    &+\frac{1}{2} \sum_{k=1}^{N} \int_{0}^{t \wedge \tau_{\ell}} \left(|f_{k}|^{2}+ \langle f_{k}, u(p)\rangle^{2} (|u(p)|_{\mathcal{H}}^{2}-2) \right)~dp\\
    &+ \int_{0}^{t \wedge \tau_{\ell}} \left( \|  u(p)\|^{2}_{\mathcal{H}^{2}_{0}}  + 2\|    u(p)\|^{2}_{\mathcal{H}^{1}_{0}}   +\| u(p)\|^{2n}_{{\mathcal{L}}^{2n}}  \right)(|u(p)|_{\mathcal{H}}^{2}-1) ~dp, ~~~~~~~\mathbb{P}-a.s
\end{align*}

Rewrite the equation by collecting Riemann integrals together, it follows:

\begin{align}{\label{Inv_eq}}
 &\left(|u(t \wedge \tau_{\ell})|^{2}_{\mathcal{H}}-1\right) 
 = \sum_{k=1}^{N} \int_{0}^{t \wedge \tau_{\ell}} 2\langle u(p), f_{k}\rangle(|u(p)|_{\mathcal{H}}^{2}-1) ~dW_{k}(p) \\
    &+ \int_{0}^{t \wedge \tau_{\ell}} \left( 2\|  u(p)\|^{2}_{\mathcal{H}^{2}_{0}}  + 4\|    u(p)\|^{2}_{\mathcal{H}^{1}_{0}}   +2\| u(p)\|^{2n}_{{\mathcal{L}}^{2n}} 
    +3 \langle f_{k}, u(p)\rangle^{2} - |f_{k}|^{2}\right)(|u(p)|_{\mathcal{H}}^{2}-1) ~dp, ~~\mathbb{P}-a.s
\end{align}

For more effortless simplification, consider $N=1$ and the following functions.

\begin{align*}
    \eta ( t) &=  |u(t \wedge \tau_{\ell})|^{2}_{\mathcal{H}}-1\\
    a_{1}(t) & = 2\langle u(t \wedge \tau_{\ell}), f_{1}\rangle \\
    a_{2}(t) &= \left( 2\|  u(t \wedge \tau_{\ell})\|^{2}_{\mathcal{H}^{2}_{0}}  + 4\|    u(t \wedge \tau_{\ell})\|^{2}_{\mathcal{H}^{1}_{0}}   +2\| u(t \wedge \tau_{\ell})\|^{2n}_{{\mathcal{L}}^{2n}} 
    +3 \langle f_{k}, u(t \wedge \tau_{\ell})\rangle^{2} - |f_{k}|^{2}\right) \\
    F_{1}(t,  \eta ( t)) &= a_{1}(t) \eta ( t) \\
    F_{2}(t, \eta ( t)) &= a_{2}(t) \eta ( t)
\end{align*}

Therefore, the equation (\ref{Inv_eq}) becomes 

\begin{align}{\label{las_Pb}}
 \eta(t) 
 &=  \int_{0}^{t \wedge \tau_{\ell}} F_{1}(p,   \eta ( p)) ~dW_{k}(p) + \int_{0}^{t \wedge \tau_{\ell}} F_{2}(p,  \eta ( p))  ~dp, ~~~~~~~\mathbb{P}-a.s \\
 \text{and}~~~~ \eta(0) &=  |u(0)|^{2}_{\mathcal{H}}-1 = 0
\end{align}
For the wellposedness of the problem (\ref{las_Pb}), it is sufficient to prove that $F_{1}$ and $F_{2}$ are Lipschitz in the second argument (see theorem 7.7, \cite{ZB-Basic}).\\
For every $y, x \in \mathbb{R},  ~0\leq  t $  and $\Omega \ni \omega $ we have:

\begin{align*}
     | F_{1}(t,  \mathcal{X}) -   F_{1}(t,  y) | &=| a_{1}(t, \omega)x - a_{1}(t,\omega)y|=|a_{1}(t, \omega)| | x-y| \\
~~~\text{and}~~~~| F_{2}(t,  x) -   F_{2}(t,  y) | &=| a_{2}(t, \omega)x - a_{2}(t,\omega)y|=|a_{2}(t, \omega)| | x-y| 
\end{align*}
   Thus, to prove that $F_{1}$ and $F_{2}$ are Lipschitz, we require proving that the functions $|a_{1}(t, \omega) |$ and  $|a_{2}(t, \omega) |$ are bounded. Consider the function  $|a_{1}(t, \omega) |$ and by applying the Cauchy Schwartz Inequality, we can imply that:

\begin{align*}
   |a_{1}(t, \omega) | \leq |2\langle u(t \wedge \tau_{\ell}, \omega ), f_{1}\rangle| \leq 2 |\langle u(t \wedge \tau_{\ell}, \omega )|_{\mathcal{H}}|f_{1}|_{\mathcal{H}}
\end{align*} 

Since $f_{1} \in \mathcal{H}$ so $|f_{1}|_{\mathcal{H}}< C < \infty $ and  the embedding $\mathcal{V} \hookrightarrow H $ is continuous, therefore $ 2 |\langle u(t \wedge \tau_{\ell}, \omega )|_{\mathcal{H}} \leq C' \|\langle u(t \wedge \tau_{\ell}, \omega )\|_{\mathcal{V}}$, It follows that:

\begin{align*}
   |a_{1}(t, \omega) | & \leq C C' \|\langle u(t \wedge \tau_{\ell}, \omega )\|_{\mathcal{V}}\leq C_{\ell}
\end{align*} 

Therefore, $   a_{1}(t, \omega) $ is bounded. Now we will discuss the boundedness of $   a_{2}(t, \omega) $. For any $t \geq 0$ and $\Omega \ni \omega $ and using the fact of continuous embeddings $\mathcal{V} \hookrightarrow \mathcal{H} $,  $\mathcal{V} \hookrightarrow {H}^{1}_{0}$, $\mathcal{V} \hookrightarrow {H}^{2}_{0}$ and $\mathcal{V} \hookrightarrow \mathcal{\mathcal{L}}^{2n}$ we have:\\
$|a_{2}(t, \omega)|$
\begin{align*}
   &= \left|\left( 2\|  u(t \wedge \tau_{\ell}, \omega )\|^{2}_{\mathcal{H}^{2}_{0}}  + 4\|    u(t \wedge \tau_{\ell}, \omega )\|^{2}_{\mathcal{H}^{1}_{0}}   +2\| u(t \wedge \tau_{\ell}, \omega )\|^{2n}_{{\mathcal{L}}^{2n}} 
    +3 \langle f_{1}, u(t \wedge \tau_{\ell}, \omega )\rangle^{2} - |f_{1}|^{2}\right)\right | \\
    &\leq  2c_{1}^{2}\|  u(t \wedge \tau_{\ell}, \omega )\|^{2}_{\mathcal{V}}  + 4c_{2}^{2}\|    u(t \wedge \tau_{\ell}, \omega )\|^{2}_{\mathcal{V}}   +2c_{3}^{2n}\| u(t \wedge \tau_{\ell}, \omega )\|^{2n}_{\mathcal{V}} 
    +3  c_{4}|f_{1}|_{\mathcal{H}} \|u(t \wedge \tau_{\ell}, \omega )\|_{\mathcal{V}}^{2} + |f_{1}|_{\mathcal{H}}^{2}  \\
    & \leq  \left(2c_{1}^{2}  + 4c_{2}^{2}+  3  c_{4}|f_{1}|_{\mathcal{H}} \right)  \| u(t , \omega )\|^{2}_{\mathcal{V}}   +2c_{3}^{2n}\| u(t  \omega )\|^{2n}_{\mathcal{V}}  + |f_{1}|_{\mathcal{H}}^{2} \\
    &\leq \left(2c_{1}^{2}  + 4c_{2}^{2}+  3  c_{4}|f_{1}|_{\mathcal{H}} \right)  K^{2}   +2c_{3}^{2n}K^{2n}  + |f_{1}|_{\mathcal{H}}^{2}< \infty.
\end{align*}

 Thus $a_{2}(t, \omega)$ is bounded too. So, there is a unique solution $ \eta'(t)$ to the problem (\ref{las_Pb}) and it also  satisfies $\eta(0)=0$, it follows that:
\begin{equation*}
   \eta'(t)= 0\quad \text{ or }\quad \eta(t)-\eta(0)=0 \quad\text{ i.e. }\quad
   |u(t \wedge \tau_{\ell})|^{2}_{\mathcal{H}}= 1. ~~~~~ \forall t \in [0,T]
\end{equation*}
 Thus, it is proved that $u(t \wedge \tau_{\ell}) \in{M} $
\end{proof} 

We define the energy function $ \mathcal{Y} : \mathcal{V}\rightarrow \mathbb{R}$ as:
\begin{align*}
    \mathcal{Y}(u) = \frac{1}{2} \|u\|^{2}_{\mathcal{V}} + \frac{1}{2n} \|u\|^{2n}_{\mathcal{\mathcal{L}}^{2n}}, ~~~~ n \in N
\end{align*}  

Now, before proving the global solution, we will discuss the following important lemma, which plays a vital role in proving the global solution of the main stochastic evolution equation.

\begin{lemma}{\label{global_lemma}}

The energy function, defined by $ \mathcal{Y} : \mathcal{V}\rightarrow \mathbb{R}$ is $C^{2}$- class  and for any $u, p , p_{1}, p_{2}$, the following equations hold:

\begin{align}{\label{enrgy_lmma_1}}
    \left\langle \mathcal{Y}'(u), p \right\rangle &\equiv d_{u} \mathcal{Y}(p) = \left\langle u, p \right\rangle_{\mathcal{V}}  + \left\langle u^{2n-1}, p\right\rangle =\left\langle \Delta^{2} u -2 \Delta u +u + u^{2n-1} , p\right \rangle 
\end{align}

\begin{align}{\label{enrgy_lmma_2}}
\left\langle \mathcal{Y}''(u)p_{1}, p_{2} \right\rangle &\equiv  d^{2}_{u} \gamma(u) (p_{1}, p_{2})= \left\langle p_{1}, p_{2}\right\rangle _{\mathcal{V}}+ \frac{2n-1}{n} \left \langle u^{2n-2}, p_{1}p_{2}\right \rangle 
\end{align} 
Furthermore, for any $f \in \mathcal{V}$ if 
\begin{align*}
    B(u) = f- \left\langle f, u \right\rangle u, ~~~ \text{and}~~~m(u) = - \left\langle f, B (u) \right\rangle u- \left\langle f, u \right\rangle B(u)
\end{align*}
then 
\begin{align}{\label{enrgy_lmma_3}}
     \left\langle \mathcal{Y}'(u), B(u)\right\rangle &= \left\langle u, f \right\rangle_{\mathcal{V}} - \left\langle u, f \right\rangle  \left( \|u\|^{2}_{\mathcal{\mathcal{L}}^{2}} +\|\Delta u\|^{2}_{\mathcal{\mathcal{L}}^{2}} + 2\|\nabla u\|^{2}_{\mathcal{\mathcal{L}}^{2}}+ \|u\|^{2n}_{\mathcal{\mathcal{L}}^{2n}}  \right) \notag \\
    &+  \left \langle u^{2n-1}, f \right \rangle
\end{align}
\begin{align}{\label{enrgy_lmma_4}}
     \langle \mathcal{Y}'(u), -\Delta^{2}u+2 \Delta u +F(u)\rangle &=- \left|  \pi_{u}\left(-\Delta^{2} u +2 \Delta u -u - u^{2n-1} \right) \right|^{2}_{\mathcal{H}} 
\end{align}

\begin{align}{\label{enrgy_lmma_5}}
    \langle \mathcal{Y}''(u) (B(u), B(u))\rangle &=  \|B\|_{\mathcal{V}}^{2}+ \frac{2n-1}{n}\left\langle u^{2n-2},  \left(B(u)\right)^{2}\right\rangle  
\end{align}

\begin{align}{\label{enrgy_lmma_6}}
    \langle \mathcal{Y}'(u), m(u)\rangle &= \left( \|u\|^{2}_{\mathcal{\mathcal{L}}^{2}} +2\|\Delta u\|^{2}_{\mathcal{\mathcal{L}}^{2}} + \|\nabla u\|^{2}_{\mathcal{\mathcal{L}}^{2}}+ \|u\|^{2n}_{\mathcal{\mathcal{L}}^{2n}}  \right) \left[ 2 \left\langle f, u \right \rangle^{2}-|f|_{\mathcal{H}}^{2}\right] \notag \\
      &- \left\langle f, u\right\rangle \left[\left\langle u, f \right\rangle_{\mathcal{V}} + \left \langle u^{2n-1}, f \right \rangle \right] 
\end{align}  
\end{lemma}

\begin{proof}
    It has already be proven in the section $2$ that the map $ \mathcal{Y} : \mathcal{V}\rightarrow \mathbb{R}$ is $C^{2}$- class. \\
    Now we will the function from (\ref{enrgy_lmma_3}) to (\ref{enrgy_lmma_6})

Consider the equation (\ref{enrgy_lmma_3}) that is:

%\begin{align*}
   % \left\langle \mathcal{Y}'(u), p \right\rangle &\equiv d_{u} \mathcal{Y}(p) = \left\langle u, p \right\rangle_{\mathcal{V}}  + \left\langle u^{2n-1}, p\right\rangle =\left\langle \Delta^{2} u -2 \Delta u +u + u^{2n-1} , p\right \rangle 
%\end{align*} 

%\begin{align*}
     %\left\langle \mathcal{Y}'(u), p \right\rangle \equiv d_{u} \mathcal{Y}(p) &= \left\langle u, p \right\rangle_{\mathcal{V}} + \left\langle u^{2n-1}, p\right\rangle \\
     %&= \left\langle u, p \right\rangle_{\mathcal{H}^{1}_{0}} +\left\langle u, p \right\rangle_{\mathcal{H}^{2}} + \left\langle u^{2n-1}, p\right\rangle \\
    % &=\left\langle \nabla u, \nabla p \right\rangle_{\mathcal{\mathcal{L}}^{2}}+ \left\langle u, p \right\rangle +\left\langle \nabla u, \nabla p \right\rangle_{\mathcal{\mathcal{L}}^{2}} +\left\langle \Delta u, \Delta p \right\rangle_{\mathcal{\mathcal{L}}^{2}} + \left\langle u^{2n-1}, p\right\rangle \\
    % &=-2 \left\langle \Delta u,  p \right\rangle+ \left\langle  u,  p \right\rangle+ \left\langle \Delta^{2} u,  p \right\rangle+ \left\langle u^{2n-1}, p\right\rangle \\
    % &= \left\langle \Delta^{2} u -2 \Delta u +u + u^{2n-1} , p\right \rangle 
%\end{align*}%%%%%%%

%\begin{align*}
     %\left\langle \mathcal{Y}'(u), B(u)\right\rangle &= \left\langle u, f \right\rangle_{\mathcal{V}} - \left\langle u, f \right\rangle  \left( \|u\|^{2}_{\mathcal{\mathcal{L}}^{2}} +\|\Delta u\|^{2}_{\mathcal{\mathcal{L}}^{2}} + 2\|\nabla u\|^{2}_{\mathcal{\mathcal{L}}^{2}}+ \|u\|^{2n}_{\mathcal{\mathcal{L}}^{2n}}  \right) \notag +  \left \langle u^{2n-1}, f \right \rangle
%\end{align*}

\begin{align*}
     \left\langle \mathcal{Y}'(u), B(u)\right\rangle &= \left\langle u, B(u)\right\rangle_{\mathcal{V}}  + \left\langle u^{2n-1}, B(u)\right\rangle \\
     &= \left\langle u, f- \left\langle f, u \right\rangle u\right\rangle_{\mathcal{V}}  + \left\langle u^{2n-1}, f- \left\langle f, u \right\rangle u\right\rangle \\
     &= \left\langle u,  f \right\rangle_{\mathcal{V}} -  \left\langle f, u \right\rangle \left\langle u, u\right\rangle_{\mathcal{V}}  + \left\langle u^{2n-1}, f \right\rangle - \left\langle f, u \right\rangle \left\langle u^{2n-1}, u\right\rangle \\
     &=\left\langle u,  f \right\rangle_{\mathcal{V}} -  \left\langle f, u \right\rangle \|u\|^{2}_{\mathcal{V}}  + \left\langle u^{2n-1}, f \right\rangle - \left\langle f, u \right\rangle \|u\|^{2n}_{\mathcal{\mathcal{L}}^{2n}} \\
     &=\left\langle u,  f \right\rangle_{\mathcal{V}} -  \left\langle f, u \right\rangle \left( \|u\|^{2}_{\mathcal{V}}  + \|u\|^{2n}_{\mathcal{\mathcal{L}}^{2n}} \right)+ \left\langle u^{2n-1}, f \right\rangle\\
     &=\left\langle u, f \right\rangle_{\mathcal{V}} - \left\langle u, f \right\rangle  \left( \|u\|^{2}_{\mathcal{\mathcal{L}}^{2}} +\|\Delta u\|^{2}_{\mathcal{\mathcal{L}}^{2}} + 2\|\nabla u\|^{2}_{\mathcal{\mathcal{L}}^{2}}+ \|u\|^{2n}_{\mathcal{\mathcal{L}}^{2n}}  \right) \notag +  \left \langle u^{2n-1}, f \right \rangle
\end{align*}

Next, consider the equation (\ref{enrgy_lmma_4}).
Using the equation (\ref{enrgy_lmma_1}) and integration by parts, we have:
\begin{align}{\label{Put_Liptchz}}
     \left\langle \mathcal{Y}'(u), -\Delta^{2}u+2 \Delta u +F(u)\right\rangle & = \left\langle u, -\Delta^{2}u+2 \Delta u +F(u) \right\rangle_{\mathcal{V}}  + \left\langle u^{2n-1}, -\Delta^{2}u+2 \Delta u +F(u)\right\rangle.
\end{align}

%By choosing $a=1$ in the function (\ref{Projection_U_st}) and using the lemma (\ref{lipschitz_st}), 
It follows  that: 
\begin{align*}
    \pi _{u}(-\Delta^{2}u+2\Delta u -u - u^{2n-1}) &=  -  \Delta^{2}u+2 \Delta u  +F(u) 
\end{align*}
Putting the above equation in (\ref{Put_Liptchz}), we have:\\
$\left\langle \mathcal{Y}'(u),  -\Delta^{2}u+2 \Delta u +F(u)\right\rangle$
\begin{align*}
     &= \left\langle u, \pi _{u}(-\Delta^{2}u+2\Delta u -u - u^{2n-1}) \right\rangle_{\mathcal{V}} + \left\langle u^{2n-1},  \pi _{u}(-\Delta^{2}u+2\Delta u -u - u^{2n-1})\right\rangle\\
    &=\left\langle   \Delta^{2}u-2 \Delta u +u  , \pi _{u}(-\Delta^{2}u+2\Delta u -u - u^{2n-1}) \right\rangle \\
    &+ \left\langle u^{2n-1},  \pi _{u}(-\Delta^{2}u+2\Delta u -u - u^{2n-1})\right\rangle\\
&= \left\langle   \Delta^{2}u-2 \Delta u+u +u^{2n-1}  , \pi _{u}(-\Delta^{2}u+2\Delta u -u - u^{2n-1}) \right\rangle \\
&= -\left\langle   -\Delta^{2}u+2 \Delta u-u -u^{2n-1}  , \pi _{u}(-\Delta^{2}u+2\Delta u -u - u^{2n-1}) \right\rangle \\
&= -\left\langle   \pi _{u}(-\Delta^{2}u+2\Delta u -u - u^{2n-1}) , \pi _{u}\left(-\Delta^{2}u+2\Delta u -u - u^{2n-1}\right) \right\rangle \\
&=-\left| \pi _{u}\left(-\Delta^{2}u+2\Delta u -u - u^{2n-1}\right)\right|_{\mathcal{H}}^{2}
\end{align*}

For the equation (\ref{enrgy_lmma_5}), using (\ref{enrgy_lmma_2}) we have:
\begin{align*}
    \mathcal{Y}''(u) (B(u), B(u)) &=  \left\langle B(u), B(u)\right\rangle _{\mathcal{V}}+ \frac{2n-1}{n} \left \langle u^{2n-2}, \left(B(u)\right)^{2}\right \rangle  \\
    &= \|B(u)\|^{2}_{\mathcal{V}}+ \frac{2n-1}{n} \left \langle u^{2n-2}, \left(B(u)\right)^{2}\right \rangle
\end{align*}

Finally, turn to the last equation (\ref{enrgy_lmma_6}).\\

\begin{align*}
  \left \langle \mathcal{Y}'(u), m_{k}(u) \right \rangle   &= \left \langle u, m_{k}(u)\right \rangle _{\mathcal{V}} + \left \langle u^{2n-1}, m_{k}(u)\right \rangle \\
     &= \left \langle u, - \left\langle f_{k}, B_{k}(u) \right\rangle u- \left\langle f_{k}, u \right\rangle B_{k}(u)\right \rangle _{\mathcal{V}} + \left \langle u^{2n-1}, - \left\langle f_{k}, B_{k} (u) \right\rangle u- \left\langle f_{k}, u \right\rangle B_{k}(u)\right \rangle \\
     &= -\left\langle f_{k}, B_{k}(u)\right\rangle \left\langle u, u\right\rangle_{\mathcal{V}}- \left\langle f_{k}, u\right\rangle\left\langle u, B_{k}(u)\right\rangle_{\mathcal{V}} - \left\langle f_{k}, B_{k}(u)\right\rangle \left\langle u^{2n-1}, u\right\rangle- \left\langle f_{k}, u\right\rangle \left\langle u^{2n-1}, B_{k}(u)\right\rangle \\
     &= -\left\langle f_{k}, B_{k}(u)\right\rangle\left( \|u\|_{\mathcal{V}}^{2}+ \|u\|_{\mathcal{\mathcal{L}}^{2n}}^{2n}\right)- \left\langle f_{k}, u\right\rangle \left(\left\langle u, B_{k}(u)\right\rangle_{\mathcal{V}} + \left\langle u^{2n-1}, B_{k}(u)\right\rangle \right)  \\ 
     &= -\left\langle f_{k}, f_{k}- \left\langle f_{k}, u \right\rangle u\right\rangle\left( \|u\|_{\mathcal{V}}^{2}+ \|u\|_{\mathcal{\mathcal{L}}^{2n}}^{2n}\right)- \left\langle f_{k}, u\right\rangle \left\langle u, f_{k}- \left\langle f_{k}, u \right\rangle u\right\rangle_{\mathcal{V}} + \left\langle f_{k}, u\right\rangle\left\langle u^{2n-1}, f_{k}- \left\langle f_{k}, u \right\rangle u \right\rangle  \\
&= - \left(\left\langle f_{k}, f_{k} \right\rangle - \left\langle f_{k}, u \right\rangle^{2} \right)  \left( \|u\|_{\mathcal{V}}^{2}+ \|u\|_{\mathcal{\mathcal{L}}^{2n}}^{2n}\right)- \left\langle f_{k}, u\right\rangle \left[\left\langle u, f_{k}\right\rangle_{\mathcal{V}} - \left\langle f_{k}, u\right\rangle \left\langle u, u\right\rangle _{\mathcal{V}} \right] \\
&~~~~~+ \left\langle f_{k}, u\right\rangle \left[ \left\langle u^{2n-1}, f_{k} \right \rangle - \left\langle f_{k}, u \right \rangle \left\langle u^{2n-1}, u \right\rangle \right] \\
      &= \left( \|u\|^{2}_{\mathcal{\mathcal{L}}^{2}} +2\|\Delta u\|^{2}_{\mathcal{\mathcal{L}}^{2}} + \|\nabla u\|^{2}_{\mathcal{\mathcal{L}}^{2}}+ \|u\|^{2n}_{\mathcal{\mathcal{L}}^{2n}}  \right) \left[ 2 \left\langle f_{k}, u \right \rangle^{2}-|f_{k}|_{\mathcal{H}}^{2}\right] - \left\langle f_{k}, u\right\rangle \left[\left\langle u, f_{k} \right\rangle_{\mathcal{V}} + \left \langle u^{2n-1}, f_{k} \right \rangle \right] 
\end{align*}

Hence, the proof is now completed.

\end{proof}

\textbf{ Proof of the global solution to the main stochastic evolution equation} \\

In the next theorem, we aim to prove the global solution of our main stochastic evolution equation (\ref{main_eq_st_Ito}). The stopping time is 
\begin{align*}
    \tau_{\ell}  = \inf \left\{ t\in [0,t] ; \|u\|_{\mathcal{V}} \geq \ell\right\}
\end{align*}

\begin{theorem}
    Suppose that the assumptions (\ref{Ass_2.2.2_St}) and the conditions of the lemma (\ref{global_lemma}) are satisfied. Then for each $u_{0} \in {M} $ - the square-integrable and   $F_{0}$- measurable random variable taking values in $\mathcal{V}$-  there is the unique global solution to the problem (\ref{main_eq_st_Ito}).
\end{theorem} 

\begin{proof}
    We have proved that there is a local maximal solution $ ( u(t), t \in [0,\tau)) $ of the problem (\ref{main_eq_st_Ito}) that satisfies $ \lim_{\tau \to \infty}\|u\|_{\mathcal{V}} = \infty , \mathbb{P} a.s.$ on $\{\tau < \infty\}$.\\
    Now, to prove the global solution to the main problem, we consider the  Khashminskii test for non-explosion (Theorem 1.1 of \cite{ZB-beam}, page 7)  and (See Theorem III.4.1 of for the finite-dimensional case) and it is sufficient to show the following conditions:

\begin{align*}
   &i) ~~~~ \mathcal{Y} \geq 0 ~~\text{on}~~~ \mathcal{V}\\
   &ii)~~~~ q_{P}:= \inf_{\|u\|_{\mathcal{V}} \geq P}{ \mathcal{Y}(u) \to\infty ~~as~~~P \to \infty} \\
   &iii) ~~~~\mathcal{Y}\left(u(0)\right) < \infty\\
&iv) ~~~~\text{For~ any ~} t>0 \text{,there ~is~ a ~ }C_{t} ~~ \text{ such ~that } \\
    &.~~~~~~~~~~~~~\mathbb{E} \left( \mathcal{Y} \left( u(t \wedge t _{\ell}\right)\right) \leq C_{t}, ~~~~ \forall \ell \in \mathbb{N}
    \end{align*}

Also, recall that:
\begin{align}{\label{Eng_inq}}
    \|u\|_{\mathcal{V}}^{2} &\leq 2 ~\mathcal{Y}(u) \\
    \|u\|_{\mathcal{\mathcal{L}}^{2n}}^{2n} &\leq 2n ~\mathcal{Y}(u){\label{Eng_inq_1}}
\end{align}
Now we prove the inequalities i)-iv)\\

 i) As by the definition of $\mathcal{Y}$ , we have:\\
 
 $\mathcal{Y}(u) = \frac{1}{2} \|u\|^{2}_{{V}} + \frac{1}{2n} \|u\|^{2n}_{\mathcal{\mathcal{L}}^{2n}} \geq 0$\\

 ii) Now, if $u \in \mathcal{V}$ such that $\|u\| \geq P$ then by using  (\ref{Eng_inq}), it follows:
 \begin{align*}
   \mathcal{Y}(u) &\geq \frac{1}{2}   \|u\|_{\mathcal{V}}^{2}  \geq \frac{P^{2}}{2} \to \infty, ~~~\text{if} ~~~~~~P \to \infty \\
q_{P}&:= \inf_{\|u\|_{\mathcal{V}} \geq P}{ \mathcal{Y}(u) \to\infty ~~as~~~P \to \infty}
 \end{align*}

iii) Using the continuous embedding $ V \hookrightarrow \mathcal{\mathcal{L}}^{2n}$, we have:
\begin{align*}
    \mathcal{Y}(u_{0}) &= \frac{1}{2} \|u_{0}\|^{2}_{{V}} + \frac{1}{2n} \|u_{0}\|^{2n}_{\mathcal{\mathcal{L}}^{2n}}\\
    & \leq \frac{1}{2} \|u_{0}\|^{2}_{\mathcal{V}} + \frac{1}{2n} \|u_{0}\|^{2n}_{\mathcal{V}} < \infty
\end{align*}
This proves iii).
Now, let's prove the final inequality iv).\\

By using Itô's lemma to,  $\mathcal{Y} \left(u(t \wedge \tau_{\ell})\right)$ we have:

\begin{align}{\label{Ito-lemma-on energy}}
    \mathcal{Y} \left(u(t \wedge \tau_{\ell})\right) -\mathcal{Y}(u_{0})&= \sum_{k=1}^{N} \int_{0}^{t \wedge \tau_{\ell}} \langle \mathcal{Y}'(u(p)) , B_{k}(u(p)) \rangle ~dW_{k}(p)\notag \\
    &+\frac{1}{2} \sum_{k=1}^{N} \int_{0}^{t \wedge \tau_{\ell}} \langle \mathcal{Y}'(u(p)) , m_{k}(u(p)) \rangle ~dp \notag \\
    &+\frac{1}{2} \sum_{k=1}^{N} \int_{0}^{t \wedge \tau_{\ell}} \mathcal{Y}''(u(p))\left( B_{k}(u(p)) , B_{k}(u(p)) \right) ~dp \notag\\ 
    &+ \int_{0}^{t \wedge \tau_{\ell}} \langle \mathcal{Y}'(u(p)) , -\Delta^{2}u(p)+2 \Delta u(p)+F(u(p)) \rangle ~dp, ~~~~~~~\mathbb{P}-a.s. ~~ \forall t \in [0,T] \notag\\
    &= \sum_{k=1}^{N} I_{1,k} + \sum_{k=1}^{N} I_{2,k}+ \sum_{k=1}^{N} I_{3,k} + I_{4} 
\end{align}

    We want to show that $I_{1,k}$ is a martingale. It is sufficient to prove the following inequality to show that it is indeed a martingale.

\begin{align*}
    \mathbb{E} \left(\int_{0}^{T \wedge \tau_{\ell}} \langle \mathcal{Y}'(u(p)) , B_{k}(u(p)) \rangle^{2} dp\right) < \infty 
\end{align*}
    Using the fact $ (t+s)^{2}\leq 2(t^{2}+s^{2})$, and (\ref{enrgy_lmma_1}) it follows that:\\
    $\mathbb{E} \left(\int_{0}^{T \wedge \tau_{\ell}} \langle \mathcal{Y}'(u(p)) , B_{k}(u(p)) \rangle^{2} dp\right)$
\begin{align*} 
&=\mathbb{E} \left(\int_{0}^{T \wedge \tau_{\ell}} \left(\langle u(p) , B_{k}(u(p)) \rangle_{\mathcal{V}}+ \langle u^{2n-1}, B_{k}(u(p)) \rangle \right)^{2} dp\right) \\
    &= \mathbb{E} \left( \int^{T \wedge \tau_{\ell}}_{0}{\langle u(p) , B_{k}(u(p)) \rangle^{2}_{\mathcal{V}}}~dp\right)+ \mathbb{E} \left( \int^{T \wedge \tau_{\ell}}_{0}{\langle u^{2n-1}, B_{k}(u(p)) \rangle^{2} }~dp\right)\\
    &+ 2\mathbb{E} \left( \int^{T\wedge \tau_{\ell}}_{0}{\langle u(p) , B_{k}(u(p)) \rangle ~\langle u^{2n-1}, B_{k}(u(p)) \rangle}~dp\right) \\
    &\leq 2\mathbb{E} \left( \int^{T\wedge \tau_{\ell}}_{0}{\left\|u(p)\right\|^{2}_{\mathcal{V}} ~\left\|B_{k}(u(p))\right\|^{2}_{\mathcal{V}}}~dp\right) \\
    &+ 2\mathbb{E} \left( \int^{T\wedge \tau_{\ell}}_{0}{\left|u(p)^{2n-1}\right|^{2}_{\mathcal{H}} ~\left|B_{k}(u(p))\right|^{2}_{\mathcal{H}}}~dp\right) 
\end{align*}
    As we know that:
\begin{align*}
    \left|u(p)^{2n-1}\right|^{2}_{\mathcal{H}} = \int_{D}{u(p)^{4n-2}}~dp= \|u\|^{4n-2}_{\mathcal{\mathcal{L}}^{4n-2}}
\end{align*}

By applying the continuous embedding $ V \hookrightarrow \mathcal{\mathcal{L}}^{4n-2}$, we have:

\begin{align*}
    \left|u(p)^{2n-1}\right|^{2}_{\mathcal{H}} \leq C^{4n-2} \|u\|^{4n-2}_{\mathcal{V}}
\end{align*}

Therefore,
\begin{align*}
    &\mathbb{E} \left(\int_{0}^{T \wedge \tau_{\ell}} \langle \mathcal{Y}'(u(p)) , B_{k}(u(p)) \rangle^{2} dp\right) \\
     &\leq 2 C^{2} \|f_{k}\|_{\mathcal{V}}^{2}\mathbb{E} \left( \int^{T\wedge \tau_{\ell}}_{0}{\left\|u(p)\right\|^{4}_{\mathcal{V}}}~~dp\right) 
     + 2 C^{2} C^{4n-2} \|f_{k}\|_{\mathcal{V}}^{2}\mathbb{E} \left( \int^{T\wedge \tau_{\ell}}_{0}{\left\|u(p)\right\|^{4n}_{\mathcal{V}}}~~dp\right) 
\end{align*}

By definition, we know that $\|u(p)\| \leq \ell ~~ \forall p \leq \tau_{\ell}$, it follows:

\begin{align*}
    \mathbb{E} \left(\int_{0}^{T \wedge \tau_{\ell}} \langle \mathcal{Y}'(u(p)) , B_{k}(u(p)) \rangle^{2} dp\right) 
     &\leq2K^{4} C^{2} \|f_{k}\|_{\mathcal{V}}^{2} \left(T \wedge \tau_{\ell}\right) + 2 k^{4n}C^{2} C^{4n-2} \|f_{k}\|_{\mathcal{V}}^{2}\left(T \wedge \tau_{\ell}\right) < \infty
\end{align*}

Thus, the integral $I_{1,k}$ is martingale and

\begin{align}{\label{I-1,kintegral}}
    \mathbb{E} \left(I_{1,k}\right) =0
\end{align}

Examine the integral $I_{2,k}$.\\
By using (\ref{enrgy_lmma_6}) and the continuous embedding $ V \hookrightarrow H$, we have,\\
$\langle \mathcal{Y}'(u), m_{k}(u)\rangle$
\begin{align*}
     &= \left( \|u\|^{2}_{\mathcal{\mathcal{L}}^{2}} +2\|\Delta u\|^{2}_{\mathcal{\mathcal{L}}^{2}} + \|\nabla u\|^{2}_{\mathcal{\mathcal{L}}^{2}}+ \|u\|^{2n}_{\mathcal{\mathcal{L}}^{2n}}  \right) \left[ 2 \left\langle f_{k}, u \right \rangle^{2}-|f_{k}|_{\mathcal{H}}^{2}\right] \\
    & - \left\langle f_{k}, u\right\rangle \left[\left\langle u, f_{k} \right\rangle_{\mathcal{V}} + \left \langle u^{2n-1}, f_{k} \right \rangle \right] \\
    &\leq \left( \|u\|^{2}_{\mathcal{V}} + \|u\|^{2n}_{\mathcal{\mathcal{L}}^{2n}}  \right) \left[ 2 |f_{k}|^{2}_{\mathcal{H}}~|u|^{2}_{\mathcal{H}} - |f_{k}|_{\mathcal{H}}^{2}\right] + |f_{k}|_{\mathcal{H}}~|u|_{\mathcal{H}}\left[ \|u\|_{\mathcal{V}} \|f\|_{\mathcal{V}}+ |f_{k}|_{\mathcal{H}}~|u^{2n-1}|_{\mathcal{H}}\right] \\
    &\leq |f_{k}|^{2}_{\mathcal{V}}\left( \|u\|^{2}_{\mathcal{V}} + \|u\|^{2n}_{\mathcal{\mathcal{L}}^{2n}}  \right) \left[ 2 ~|u|^{2}_{\mathcal{H}} - 1\right] +c^{2}|f_{k}|^{2}_{\mathcal{V}}\left[ \|u\|^{2}_{\mathcal{V}}+|u^{2n-1}|_{\mathcal{H}} |u|_{\mathcal{H}} \right]
    \end{align*}
Now, we know that:
\begin{align*}
    |u^{2n-1}|^{2}_{\mathcal{H}} = \int_{D}{\left(u(p)\right)^{2n-1}}~dp
\end{align*}

By using Holder's Inequality and choosing $\frac{1}{p}=\frac{4n-2}{2n}$ and $\frac{1}{q}= \frac{1-n}{n}$, we have:

\begin{align*}
      |u^{2n-1}|^{2}_{\mathcal{H}} &\leq \left(\int_{D}\left({\left(u(p)\right)^{4n-2}}\right)^{\frac{2n}{4n-2}}~dp\right)^{\frac{4n-2}{2n}} ~\left(\int_{D} 1 ~dp\right)^{\frac{1-n}{n}} = \hat{C}^{2} \|u\|_{\mathcal{\mathcal{L}}^{2n}}^{4n-2} \\
      &\leq \hat{C} \|u\|_{\mathcal{\mathcal{L}}^{2n}}^{4n-2} \\
      &\leq  \hat{C} \left\{1, \|u\|_{\mathcal{\mathcal{L}}^{2n}}^{4n-2}\right\} \\
     |u^{2n-1}|^{2}_{\mathcal{H}}  &\leq  \hat{C} \left(1+\|u\|_{\mathcal{\mathcal{L}}^{2n}}^{4n-2}\right)
\end{align*}
 Where $\left(\int_{D} 1 ~dp\right)^{\frac{1-n}{n}}:= \hat{C}^{2} < \infty$\\
 
Therefore, the integral $I_{2,k}$ becomes.

\begin{align*}
    \int_{0}^{t \wedge \tau_{\ell}}{ \langle \mathcal{Y}'(u), m_{k}(u)\rangle }~dp & \leq  \int_{0}^{t \wedge \tau_{\ell}}{|f_{k}|^{2}_{\mathcal{V}}\left( \|u\|^{2}_{\mathcal{V}} + \|u\|^{2n}_{\mathcal{\mathcal{L}}^{2n}}  \right) \left[ 2 ~|u|^{2}_{\mathcal{H}} - 1\right]}\\
    &+{c^{2}|f_{k}|^{2}_{\mathcal{V}}\left[ \|u\|^{2}_{\mathcal{V}}+\hat{C} \left(1+\|u\|_{\mathcal{\mathcal{L}}^{2n}}^{4n-2} |u|_{\mathcal{H}}\right) |u|_{\mathcal{H}}\right] } ~dp 
\end{align*}

By applying  the invariance of manifold, that is $u(t) \in {M}$ we have:

    \begin{align*}
    \int_{0}^{t \wedge \tau_{\ell}}{ \langle \mathcal{Y}'(u), m_{k}(u)\rangle }~dp & \leq  \int_{0}^{t \wedge \tau_{\ell}}{|f_{k}|^{2}_{\mathcal{V}}\left( \|u\|^{2}_{\mathcal{V}} + \|u\|^{2n}_{\mathcal{\mathcal{L}}^{2n}}  \right) }
    +{c^{2}|f_{k}|^{2}_{\mathcal{V}}\left[ \|u\|^{2}_{\mathcal{V}}+\hat{C} \left(1+\|u\|_{\mathcal{\mathcal{L}}^{2n}}^{4n-2} \right) \right] } ~dp \\
    &\leq {c^{2}|f_{k}|^{2}_{\mathcal{V}}} \int_{0}^{t \wedge \tau_{\ell}}{\left[ 2 \|u\|^{2}_{\mathcal{V}} + \hat{C} +  \left(1+\hat{C} \right) \|u\|_{\mathcal{\mathcal{L}}^{2n}}^{4n-2}\right]}~dp
\end{align*} 
Now by using the inequalities (\ref{Eng_inq}) and (\ref{Eng_inq_1}), it follows that:

\begin{align*}
    \int_{0}^{t \wedge \tau_{\ell}}{ \langle \mathcal{Y}'(u), m_{k}(u)\rangle }~dp &\leq {c^{2}|f_{k}|^{2}_{\mathcal{V}}} \int_{0}^{t \wedge \tau_{\ell}}{\left[ 4 \mathcal{Y}\left(u(p)\right) + \hat{C} + 2n \left(1+\hat{C} \right) \mathcal{Y}\left(u(p)\right)\right]}~dp\\
    & \leq {c^{2}|f_{k}|^{2}_{\mathcal{V}}} \left(4  + 2n \left(1+\hat{C} \right) \right) \int_{0}^{t \wedge \tau_{\ell}}{ \mathcal{Y}\left(u(p)\right)}~dp +  {c^{2}|f_{k}|^{2}_{\mathcal{V}}}\hat{C}  \left(t \wedge \tau_{\ell}\right) \\
    \end{align*}
\begin{align}{\label{I-2,kintegral}}
    \int_{0}^{t \wedge \tau_{\ell}}{ \langle \mathcal{Y}'(u), m_{k}(u)\rangle }~dp &\leq C_{1} \int_{0}^{t \wedge \tau_{\ell}}{ \mathcal{Y}\left(u(p)\right)}~dp +  C_{2}  \left(t \wedge \tau_{\ell}\right)
\end{align}
Where $C_{1} ={c^{2}|f_{k}|^{2}_{\mathcal{V}}} \left(4  + 2n \left(1+\hat{C} \right) \right) $ and $C_{2}={c^{2}|f_{k}|^{2}_{\mathcal{V}}}\hat{C} $ \\

Now consider the integral $I_{3,k}$. From the inequality (\ref{enrgy_lmma_5}), it follows:

\begin{align*}
     \mathcal{Y}''(u) (B_{k}(u), B_{k}(u)) &=  \|B_{k}(u)\|_{\mathcal{V}}^{2}+ \left(\frac{2n-1}{n} \right)\left\langle u^{2n-2},  \left(B_{k}(u)\right)^{2}\right\rangle\\
     & \leq \| f_{k}- \left\langle f_{k}, u \right\rangle u\|_{\mathcal{V}}^{2}+ \left(\frac{2n-1}{n} \right)\left\langle u^{2n-2},  \left(B_{k}(u)\right)^{2}\right\rangle\\
     & \leq\left( \| f_{k}\|_{\mathcal{V}}+ |f_{k}|_{\mathcal{H}} |u|_{\mathcal{H}} \|u\|_{\mathcal{V}}\right)^{2}+ \left(\frac{2n-1}{n} \right)\left\langle u^{2n-2},  \left(B_{k}(u)\right)^{2}\right\rangle\\
\end{align*} 
\begin{align}{\label{I-3k-ineq}}
      \mathcal{Y}''(u) (B_{k}(u), B_{k}(u)) &\leq \left( \| f_{k}\|_{\mathcal{V}}+ |f_{k}|_{\mathcal{H}} |u|_{\mathcal{H}} \|u\|_{\mathcal{V}}\right)^{2}+ \left(\frac{2n-1}{n} \right)\left\langle u^{2n-2},  \left(B_{k}(u)\right)^{2}\right\rangle
\end{align}

Using the basic inequality $(p-q)^{2} \leq 2(p^{2}+q^{2})$ to the $\left\langle u^{2n-2},  \left(B_{k}(u)\right)^{2}\right\rangle$, it follows that:
\begin{align*}
    \left\langle u^{2n-2},  \left(B_{k}(u)\right)^{2}\right\rangle &= \left\langle u^{2n-2},  \left(f_{k}- \left\langle f_{k}, u \right\rangle u \right)^{2}\right\rangle\\
    &\leq \left\langle u^{2n-2},  f^{2}_{k}+ \left\langle f_{k}, u \right\rangle^{2} u^{2} \right\rangle\\
    &\leq 2\left\langle u^{2n-2},  f^{2}_{k} \right\rangle +2   |f_{k}|^{2}_{\mathcal{H}}~ |u|^{2}_{\mathcal{H}} \|u\|^{2n}_{\mathcal{\mathcal{L}}^{2n}}  \\
    &\leq 2\int_{D} {u(p)^{2n-2},  f^{2}_{k}(p) } ~dp +2   |f_{k}|^{2}_{\mathcal{H}}~ |u|^{2}_{\mathcal{H}} \|u\|^{2n}_{\mathcal{\mathcal{L}}^{2n}}
\end{align*}
By using Holder's Inequality and choosing $\frac{1}{p}=\frac{n-1}{n}$ and $\frac{1}{q}= \frac{1}{n}$, we have: 
\begin{align*}
    \left\langle u^{2n-2},  \left(B_{k}(u)\right)^{2}\right\rangle &\leq 2 \left( \int_{D}{u(p)^{2n}}\right)^{\frac{n-1}{n}}~\left( \int_{D}{f_{k}(p)^{2n}}\right)^{\frac{1}{n}}+2   |f_{k}|^{2}_{\mathcal{H}}~ |u|^{2}_{\mathcal{H}} \|u\|^{2n}_{\mathcal{\mathcal{L}}^{2n}} \\
    &\leq 2 \|u\|^{2n-2}_{\mathcal{\mathcal{L}}^{2n}}|f_{k}|^{2}_{\mathcal{\mathcal{L}}^{2n}}+ 2   |f_{k}|^{2}_{\mathcal{H}}~ |u|^{2}_{\mathcal{H}} \|u\|^{2n}_{\mathcal{\mathcal{L}}^{2n}} \\
    &\leq 2c^{2} \|u\|^{2n-2}_{\mathcal{\mathcal{L}}^{2n}}~\|f_{k}\|^{2}_{\mathcal{V}}+ 2   |f_{k}|^{2}_{\mathcal{H}}~ |u|^{2}_{\mathcal{H}} \|u\|^{2n}_{\mathcal{\mathcal{L}}^{2n}} \\
    &\leq 2c^{2}~~max \left\{\|u\|^{2n-2}_{\mathcal{\mathcal{L}}^{2n}}, 1 \right\}~\|f_{k}\|^{2}_{\mathcal{V}}+ 2   |f_{k}|^{2}_{\mathcal{H}}~ |u|^{2}_{\mathcal{H}} \|u\|^{2n}_{\mathcal{\mathcal{L}}^{2n}} \\
    &\leq 2c^{2} \left\{\|u\|^{2n-2}_{\mathcal{\mathcal{L}}^{2n}}+ 1 \right\}~\|f_{k}\|^{2}_{\mathcal{V}}+ 2   |f_{k}|^{2}_{\mathcal{H}}~ |u|^{2}_{\mathcal{H}} \|u\|^{2n}_{\mathcal{\mathcal{L}}^{2n}} \\
    &\leq  2c^{2} \|f_{k}\|^{2}_{\mathcal{V}}+ \left( 2c^{2} \|f_{k}\|^{2}_{\mathcal{V}} + 2 |f_{k}|^{2}_{\mathcal{H}}~ |u|^{2}_{\mathcal{H}} \right)\|u\|^{2n}_{\mathcal{\mathcal{L}}^{2n}} \\
\end{align*}
Using this inequality to (\ref{I-3k-ineq}), it follows:

\begin{align*}
      \mathcal{Y}''(u) (B_{k}(u), B_{k}(u)) &\leq  \left( \| f_{k}\|_{\mathcal{V}}+ |f_{k}|_{\mathcal{H}} |u|_{\mathcal{H}} \|u\|_{\mathcal{V}}\right)^{2}\\
      &+ \left(\frac{2n-1}{n} \right)\left[2c^{2} \|f_{k}\|^{2}_{\mathcal{V}}+ \left( 2c^{2} \|f_{k}\|^{2}_{\mathcal{V}} + 2 |f_{k}|^{2}_{\mathcal{H}}~ |u|^{2}_{\mathcal{H}} \right)\|u\|^{2n}_{\mathcal{\mathcal{L}}^{2n}}\right]
    \end{align*} 

Therefore, the integral $I_{3,k}$ takes the form of:

\begin{align*}
    I_{3,k}&= \int^{t \wedge \tau_{\ell}}_{0}{\mathcal{Y}''(u(p)) (B_{k}(u(p)), B_{k}(u(p)))}~dp\\
    &\leq \int^{t \wedge \tau_{\ell}}_{0}{ \left( \| f_{k}\|_{\mathcal{V}}+C \|f_{k}\|_{\mathcal{V}} |u(p)|_{\mathcal{H}} \|u(p)\|_{\mathcal{V}}\right)^{2}}~dp\\
      &+ \int^{t \wedge \tau_{\ell}}_{0}\left[\left(\frac{2n-1}{n} \right)\left\{2c^{2} \|f_{k}\|^{2}_{\mathcal{V}}+ \left( 2c^{2} \|f_{k}\|^{2}_{\mathcal{V}} + 2C^{2} \|f_{k}\|^{2}_{\mathcal{V}}~ |u|^{2}_{\mathcal{H}} \right)\|u\|^{2n}_{\mathcal{\mathcal{L}}^{2n}}\right\}\right]~dp
\end{align*}
Using the invariance of the manifold, that is  $u(t) \in {M}$, we have:

\begin{align*}
    I_{3,k} &\leq \int^{t \wedge \tau_{\ell}}_{0}{ \|f_{k}\|_{\mathcal{V}}^{2}\left(  1+  C\right)^{2}\|u(p)\|_{\mathcal{V}}^{2}}~dp\\
      &+ \int^{t \wedge \tau_{\ell}}_{0}\left[\left(\frac{2n-1}{n} \right)\left\{2c^{2} \|f_{k}\|^{2}_{\mathcal{V}}+ \left( 2c^{2}  + 2C^{2} ~  \right) \|f_{k}\|^{2}_{\mathcal{V}}\|u\|^{2n}_{\mathcal{\mathcal{L}}^{2n}}\right\}\right]~dp
\end{align*}

Now, by using the inequalities (\ref{Eng_inq}) and (\ref{Eng_inq_1}), it follows that:

\begin{align*}
    I_{3,k} &\leq \int^{t \wedge \tau_{\ell}}_{0}{ \|f_{k}\|_{\mathcal{V}}^{2}\left(  1+  C\right)^{2}\mathcal{Y}\left(u(p)\right)}~dp+ \int^{t \wedge \tau_{\ell}}_{0}\left[\left(\frac{2n-1}{n} \right)\left\{2c^{2} \|f_{k}\|^{2}_{\mathcal{V}}+ \left( 2c^{2}  + 2C^{2} ~  \right) \|f_{k}\|^{2}_{\mathcal{V}}\mathcal{Y}\left(u(p)\right)\right\}\right]~dp\\
      &= \int^{t \wedge \tau_{\ell}}_{0}{ \left[\left(\|f_{k}\|_{\mathcal{V}}^{2}\left(  1+  C\right)^{2} + \left(\frac{2n-1}{n} \right) \left( 2c^{2}  + 2C^{2} ~  \right) \|f_{k}\|^{2}_{\mathcal{V}}\right)\mathcal{Y}\left(u(p)\right)+ 2c^{2}\left(\frac{2n-1}{n} \right)  \|f_{k}\|^{2}_{\mathcal{V}}\right] }~dp
\end{align*} 

\begin{align}{\label{I-3,kintegral}}
     I_{3,k} &\leq C_{3}\int^{t \wedge \tau_{\ell}}_{0}{ \mathcal{Y}\left(u(p)\right) }~dp+ C_{4} \left(t \wedge \tau_{\ell}\right)
\end{align}

Where $C_{3}=\left(\|f_{k}\|_{\mathcal{V}}^{2}\left(  1+  C\right)^{2} + \left(\frac{2n-1}{n} \right) \left( 2c^{2}  + 2C^{2} ~  \right) \|f_{k}\|^{2}_{\mathcal{V}}\right) < \infty$ and $C_{4}=2c^{2}\left(\frac{2n-1}{n} \right)  \|f_{k}\|^{2}_{\mathcal{V}} < \infty$.\\

Now consider the last integral $I_{4}$. By using (\ref{enrgy_lmma_4}), it follows that:

\begin{align}{\label{I-4integral}}
    I_{4} &= \int_{0}^{t \wedge \tau_{\ell}} \langle \mathcal{Y}'(u(p)) , -\Delta^{2}u(p)+2 \Delta u(p)+F(u(p)) \rangle ~dp \\
    &= -\int_{0}^{t \wedge \tau_{\ell}}{\left|  \pi_{u}\left(-\Delta^{2} u(p) +2 \Delta u(p) -u(p) - u^{2n-1}(p) \right) \right|^{2}_{\mathcal{H}}}~dp \notag
\end{align}

Now using the inequalities (\ref{I-2,kintegral}), (\ref{I-3,kintegral}) and ((\ref{I-4integral})) into (\ref{Ito-lemma-on energy}), it follows:

\begin{align*}
      \mathcal{Y} \left(u(t \wedge \tau_{\ell})\right) -\mathcal{Y}(u_{0})&\leq  \sum_{k=1}^{N} I_{1,k}+ \sum_{k=1}^{N} \left(C_{1} \int_{0}^{t \wedge \tau_{\ell}}{ \mathcal{Y}\left(u(p)\right)}~dp +  C_{2}  \left(t \wedge \tau_{\ell}\right)\right) \\
      &+ \sum_{k=1}^{N} \left(C_{3}\int^{t \wedge \tau_{\ell}}_{0}{ \mathcal{Y}\left(u(p)\right) }~dp+ C_{4} \left(t \wedge \tau_{\ell}\right)\right)\\
      &- \int_{0}^{t \wedge \tau_{\ell}}{\left|  \pi_{u}\left(-\Delta^{2} u(p) +2 \Delta u(p) -u(p) - u^{2n-1}(p) \right) \right|^{2}_{\mathcal{H}}}~dp \\
      &\leq \sum_{k=1}^{N} I_{1,k}+ N C_{1} \int_{0}^{t \wedge \tau_{\ell}}{ \mathcal{Y}\left(u(p)\right)}~dp +  NC_{2}  \left(t \wedge \tau_{\ell}\right) + NC_{3}\int^{t \wedge \tau_{\ell}}_{0}{ \mathcal{Y}\left(u(p)\right) }~dp \\
      &+ NC_{4} \left(t \wedge \tau_{\ell}\right) \\
      &\leq \sum_{k=1}^{N} I_{1,k} \left(N C_{1}+NC_{3}\right)\int_{0}^{t \wedge \tau_{\ell}}{ \mathcal{Y}\left(u(p)\right)}~dp +  \left(NC_{3}+NC_{4}\right) \left(t \wedge \tau_{\ell}\right)
\end{align*}

\begin{align}
    \mathcal{Y} \left(u(t \wedge \tau_{\ell})\right) -\mathcal{Y}(u_{0})&\leq \sum_{k=1}^{N} I_{1,k} +  C_{5}\int_{0}^{t \wedge \tau_{\ell}}{ \mathcal{Y}\left(u(p)\right)}~dp +  C_{6} \left(t \wedge \tau_{\ell}\right)
\end{align}
Where $C_{5}=\left(N C_{1}+NC_{3}\right) < \infty$ and $C_{6} =\left(NC_{3}+NC_{4}\right)< \infty$

By applying the expectation to both sides, we have:

\begin{align*}
    \mathbb{E}\left( \mathcal{Y} \left(u(t \wedge \tau_{\ell})\right) \right)&\leq  \mathbb{E}\left(\mathcal{Y}(u_{0})\right)  +  C_{5} \mathbb{E}\left(\int_{0}^{t \wedge \tau_{\ell}}{ \mathcal{Y}\left(u(p)\right)}~dp\right) +  C_{6} T
\end{align*}

 By using the Gronwall's inequality, we have:

 \begin{align*}
    \mathbb{E}\left( \mathcal{Y} \left(u(t \wedge \tau_{\ell})\right) \right)&\leq  \mathbb{E}\left(\mathcal{Y}(u_{0})\right)  +   \int_{0}^{t \wedge \tau_{\ell}}{\left( \mathbb{E}\left(\mathcal{Y}(u_{0})\right) + C_{5}\left(t \wedge \tau_{\ell}\right) \cdot C_{5}exp\left(\int_{0}^{t \wedge \tau_{\ell}}C_{5}~dp\right)\right)} +  C_{6} T:= C_{t}
\end{align*}
 Thus, we have the following inequality.
 \begin{align}
      \mathbb{E}\left( \mathcal{Y} \left(u(t \wedge \tau_{\ell})\right) \right)&\leq C_{t}
 \end{align}
All four conditions of the Khashminskii test for non-explosion were satisfied.\\
Hence, $T = \infty, ~~~\mathbb{P}-a.s.$\\
\end{proof}

\end{document}